\pdfoutput=1
\documentclass{amsart}
\usepackage[mathscr]{eucal}
\usepackage{amssymb}
\usepackage[usenames,dvipsnames]{xcolor} 
\usepackage[normalem]{ulem}
\usepackage{wasysym}
\usepackage{amsthm}
\usepackage{bbold}
\usepackage{comment}
\usepackage{enumitem}
\usepackage{amsmath}
\usepackage{amssymb}
\usepackage{tikz-cd}
\usetikzlibrary{arrows}
\usepackage{stmaryrd} 
\usepackage{etoolbox} 
\usepackage{microtype}
\usepackage{adjustbox}
\usepackage[unicode]{hyperref} 
\usepackage{pdflscape}

\definecolor{dark-red}{rgb}{0.5,0.15,0.15}
\definecolor{dark-blue}{rgb}{0.15,0.15,0.6}
\definecolor{dark-green}{rgb}{0.15,0.6,0.15}

\hypersetup{
    colorlinks, linkcolor=OliveGreen,
    citecolor=OliveGreen, urlcolor=OliveGreen
}

\usepackage[nameinlink,capitalise,noabbrev]{cleveref}

\numberwithin{equation}{section}
\usepackage[all]{xy}
\xyoption{line}
\usepackage{graphicx}
\usepackage{mathtools}
\usepackage{caption}

\newcommand{\xrightarrowdbl}[2][]{%
  \xrightarrow[#1]{#2}\mathrel{\mkern-14mu}\rightarrow
}

\newtheorem{Thm}[equation]{Theorem}
\newtheorem*{Thm*}{Theorem}
\newtheorem*{MainThm*}{Main Theorem}
\newtheorem{Prop}[equation]{Proposition}
\newtheorem{Lem}[equation]{Lemma}
\newtheorem{Cor}[equation]{Corollary}

\newtheorem*{Que*}{Question}

\theoremstyle{remark}
\newtheorem{Def}[equation]{Definition}
\newtheorem{Ter}[equation]{Terminology}
\newtheorem{Not}[equation]{Notation}
\newtheorem{Exa}[equation]{Example}

\newtheorem{Rec}[equation]{Recollection}

\newtheorem{Rem}[equation]{Remark}

\tikzset{
    labelrotatebelow/.style={anchor=north, rotate=90, inner sep=1.0mm}
}
\tikzset{
    labelrotateabove/.style={anchor=south, rotate=90, inner sep=1.0mm}
}
\SelectTips{cm}{10}

\newcommand{\nc}{\newcommand}
\nc{\dmo}{\DeclareMathOperator}

\renewcommand{\emptyset}{\varnothing}

\nc{\Niko}[1]{{\color{Orange}#1}}
\nc{\Maxime}[1]{{\color{Green}#1}}
\nc{\Luca}[1]{{\color{Blue}#1}}

\usepackage{todonotes}

\nc{\overbar}[1]{\mkern 1.5mu\overline{\mkern-1.5mu#1\mkern-1.5mu}\mkern 1.5mu}
\nc{\cov}{\mathrm{cov}}
\nc{\et}{\acute{e}t}
\nc{\kappaaux}{g}
\nc{\kappaCh}{{\kappaaux(\cat C_h)}}
\nc{\kappam}{{\kappaaux({\mathfrak m})}}
\nc{\kappaP}{\Gamma_{\cat P}\unit}
\nc{\kappaQ}{{\kappaaux(\cat Q)}}
\nc{\kappaCP}{{\kappaaux_{\cat C}(\cat P)}}
\nc{\kappaDP}{{\kappaaux_{\cat D}(\cat P)}}
\nc{\kappaCQ}{{\kappaaux_{\cat C}(\cat Q)}}
\nc{\kappaDQ}{{\kappaaux_{\cat D}(\cat Q)}}
\nc{\kappaphiB}{{\kappaaux(\phi(\cat B))}}
\nc{\kappaphiQ}{{\kappaaux(\varphi(\cat Q))}}

\dmo{\Sub}{Sub}
\dmo{\tw}{tw}
\dmo{\Proj}{Proj}
\dmo{\LMod}{LMod}
\dmo{\cell}{cell}
\nc{\SpEn}{\cat S_{E(n)}}
\nc{\SpEnf}{\cat S_n}
\nc{\Lcomp}{L^{\mathrm{com}}} 
\nc{\Ucomp}{U^{\mathrm{com}}}
\nc{\Loco}[1]{\Loc_{\otimes}\hspace{-0.3ex}\langle #1 \rangle}
\nc{\bbullet}{{\scriptscriptstyle\hspace{-1pt}\bullet}}
\nc{\bullett}{{\scriptscriptstyle\bullet}\hspace{-1pt}}
\nc{\LF}{L\hspace{-0.2ex}F}
\nc{\SpG}{\Sp^G}
\nc{\EG}{\bbE_G}
\nc{\DEG}{\Der(\EG)}
\nc{\DE}{\Der(\bbE)}
\nc{\Prst}{{\cat P}\mathrm{r^{st}}}
\nc{\Mack}[2]{\mathrm{Mack}_{#1}(#2)}
\nc{\SC}{S\cat C}
\dmo{\fin}{{fin}}
\dmo{\DM}{DM}
\dmo{\fp}{fp}
\nc{\DMQ}{\DM_Q}
\dmo{\DerKal}{DMack}
\dmo{\Der}{D}
\dmo{\DMot}{DMot}
\dmo{\rmH}{H}
\dmo{\piu}{\underline{\pi}}
\dmo{\Sphere}{\mathbb{S}}
\dmo{\Alg}{Alg}
\dmo{\CAlg}{CAlg}
\nc{\HA}{{\rmH \hspace{-0.2em}\bbA}}
\nc{\HZ}{{\rmH \hspace{-0.2em}\bbZ}}
\nc{\HZbar}{{\rmH \hspace{-0.2em}\underline{\bbZ}}}
\nc{\Fp}{{\bbF_{\hspace{-0.1em}p}}}
\nc{\HFp}{{\rmH \hspace{-0.15em}\bbF_{\hspace{-0.1em}p}}}
\nc{\DHZG}{\Der(\HZ_G)}
\nc{\DHZH}{\Der(\HZ_H)}
\nc{\DHZK}{\Der(\HZ_K)}
\nc{\DHZGN}{\Der(\HZ_{G/N})}
\nc{\DHZGG}{\Der(\HZ_{G/G})}
\nc{\DHZCp}{\Der(\HZ_{C_p})}
\nc{\DHZGprime}{\Der(\HZ_{G'})}
\nc{\DHZ}{\Der(\HZ)}
\nc{\mathfrakp}{\mathfrak{p}}
\nc{\mathfrakq}{\mathfrak{q}}
\nc{\mathfrakS}{\mathfrak{S}}
\nc{\mathfrakT}{\mathfrak{T}}
\nc{\Z}{\mathbb{Z}}
\nc{\SSG}{\text{sSet}_*^G}
\nc{\sSet}{\text{sSet}}

\nc{\alg}{\mathrm{alg}}
\dmo{\csupp}{csupp}
\dmo{\Con}{Conj}
\dmo{\Id}{Id}
\dmo{\Loc}{Loc}
\dmo{\rmK}{\textrm{\rm K}}
\dmo{\Spc}{Spc}
\dmo{\thick}{thick}
\nc{\thickt}[1]{\thick_\otimes\langle #1 \rangle}
\nc{\loct}[1]{\loc_\otimes\langle #1 \rangle}
\dmo{\cone}{cone}
\dmo{\End}{End}
\dmo{\Derperf}{D_{perf}}
\dmo{\Mor}{Mor}
\dmo{\Hom}{Hom}
\dmo{\id}{id}
\dmo{\incl}{incl}
\dmo{\Img}{Im}
\dmo{\im}{im}
\dmo{\Ker}{Ker}
\dmo{\ind}{ind}
\dmo{\Ind}{Ind}
\dmo{\CoInd}{coind}
\dmo{\res}{res}
\dmo{\infl}{infl}
\dmo{\Derqc}{D_{qc}}
\dmo{\triv}{triv}
\dmo{\sep}{sep}
\dmo{\Tel}{Tel} 
\dmo{\grMod}{grMod}%
\dmo{\Mod}{Mod}%
\dmo{\opname}{op}
\dmo{\SH}{SH}
\dmo{\Shv}{Shv}
\dmo{\smallb}{b}
\dmo{\Spec}{Spec}
\dmo{\supp}{supp}
\dmo{\Supp}{Supp}
\dmo{\cosupp}{cosupp}
\dmo{\Cosupp}{Cosupp}
\nc{\SHc}{{\SH^c}}
\nc{\SHp}{{\SH_{(p)}}}
\nc{\SHcp}{{\SH^c_{(p)}}}
\nc{\SHG}{\SH(G)}
\nc{\SHGp}{\SH(G)_{(p)}}
\nc{\SHGc}{\SHG^c}
\nc{\SHGcp}{\SHG^c_{(p)}}
\nc{\quadtext}[1]{\quad\textrm{#1}\quad}
\nc{\qquadtext}[1]{\qquad\textrm{#1}\qquad}
\nc{\adj}{\dashv}
\nc{\adjto}{\rightleftarrows}
\nc{\bbL}{\mathbb{L}}
\nc{\bbA}{\mathbb{A}}
\nc{\bbE}{\mathbb{E}}
\nc{\bbN}{\mathbb{N}}
\nc{\bbQ}{\mathbb{Q}}
\nc{\bbZ}{\mathbb{Z}}
\nc{\bbF}{\mathbb{F}}
\nc{\cat}[1]{\mathscr{#1}}
\nc{\ie}{{\sl i.e.}, }
\nc{\into}{\mathop{\rightarrowtail}}
\nc{\inv}{^{-1}}
\nc{\isoto}{\mathop{\overset{\sim}\to}}
\nc{\isotoo}{\mathop{\overset{\sim}\too}}
\nc{\onto}{\mathop{\twoheadrightarrow}}
\nc{\too}{\mathop{\longrightarrow}\limits}
\nc{\mapstoo}{\longmapsto}
\nc{\adh}[1]{\overline{#1}}
\nc{\adhpt}[1]{\adh{\{#1\}}}
\nc{\aka}{{a.\,k.\,a.}\ }
\nc{\calF}{\mathcal{F}}
\nc{\eg}{{\sl e.\,g.}}
\nc{\Homcat}[1]{\Hom_{\cat #1}}
\nc{\hook}{\hookrightarrow}
\nc{\ideal}[1]{\langle #1\rangle}
\nc{\ihom}{{\underline{\hom}}}
\nc{\iHom}{\underline{\mathrm{Hom}}}
\nc{\Mid}{\,\big|\,}
\nc{\MMod}{\,\text{-}\Mod}%
\nc{\GrMMod}{\,\text{-}\grMod}%
\nc{\op}{^{\opname}}
\nc{\oto}[1]{\overset{#1}\to}
\nc{\otoo}[1]{\overset{#1}{\,\too\,}}
\nc{\sminus}{\!\smallsetminus\!}
\nc{\poplus}[1]{^{\oplus #1}}%
\nc{\potimes}[1]{^{\otimes #1}}
\nc{\sbull}{{\scriptscriptstyle\bullet}}
\nc{\SET}[2]{\big\{\,#1\Mid#2\,\big\}}
\nc{\SpcK}{\Spc(\cat K)}
\nc{\then}{\Rightarrow}
\nc{\unit}{\mathbb{1}}
\nc{\xra}{\xrightarrow}
\nc{\phigeom}[1]{\widetilde{\Phi}^{#1}}
\dmo{\Oname}{O}
\dmo{\proper}{proper}
\dmo{\lenormal}{\unlhd}
\dmo{\lnormal}{\lhd}
\nc{\normal}{\trianglelefteq}
\nc{\Op}{\Oname^p}
\nc{\Oq}{\Oname^q}
\dmo{\Sp}{Sp}
\dmo{\Ho}{Ho}
\dmo{\Fin}{Fin}
\dmo{\add}{add}
\dmo{\Fun}{Fun}
\dmo{\Ext}{Ext}

\dmo{\CMon}{CMon}
\dmo{\BB}{\cat B}
\dmo{\CC}{\cat C}
\dmo{\DD}{\cat D}
\dmo{\MM}{\cat M}
\dmo{\NN}{\cat N}
\dmo{\OO}{\mathcal{O}}
\dmo{\Map}{Map}
\dmo{\Span}{Span}
\dmo{\N}{N}
\dmo{\Cat}{Cat}
\dmo{\colim}{colim}
\dmo{\hocolim}{hocolim}
\dmo{\Ch}{Ch}
\dmo{\A}{\mathbb{A}^{eff}}
\nc{\AGeff}{\mathbb{A}_G^{\mathrm{eff}}}
\nc{\BGeff}{\mathcal{B}_G^{\mathrm{eff}}}
\nc{\BG}{{\mathcal{B}_G}}
\nc{\NBGeff}{{\N}{\BGeff}}
\dmo{\Ab}{Ab}
\dmo{\Set}{Set}
\dmo{\ev}{ev}
\dmo{\Spcl}{Spcl}
\nc{\Funadd}{\Fun_{\add}}
\dmo{\proj}{proj}
\dmo{\cof}{cof}

\nc{\StModfin}{\mathrm{StMod}^{\mathrm{fin}}}
\nc{\PrL}{\mathrm{Pr}^{\mathrm{L}}_{\mathrm{st}}}

\dmo{\Coideal}{Coideal}
\dmo{\gen}{gen}
\dmo{\Coloc}{Coloc}
\nc{\Coloco}[1]{\Coloc^{\iHom}\hspace{-0.3ex}\langle #1 \rangle}

\dmo{\dual}{dual}

\nc{\LOCO}{\mathcal{L}\mathrm{oc}_{\otimes}}
\nc{\COLOCO}{\mathcal{C}\mathrm{oloc}^{\iHom}}

\nc{\Perff}[1]{\mathrm{Perf}_{#1}}
\nc{\Modd}[1]{\mathrm{Mod}_{#1}}
\dmo{\Perf}{Perf}
\dmo{\tel}{tel}
\nc{\Lococat}[2]{\Loc^{#1}_{\otimes}\hspace{-0.3ex}\langle #2 \rangle}
\dmo{\rk}{rk}
\dmo{\StMod}{StMod}
\dmo{\stmod}{stmod}

\nc{\CSep}{\mathcal{C}\mathcal{S}\mathrm{ep}}
\nc{\CFin}{\mathcal{F}\mathrm{in}}
\nc{\Ccpl}[1]{\CC^{#1\text{-}\mathrm{cpl}}}
\nc{\Ctors}[1]{\CC^{#1\text{-}\mathrm{tors}}}
\nc{\Cloc}[1]{\CC^{#1\text{-}\mathrm{loc}}}
\nc{\fib}[1]{\mathrm{fib}}
\nc{\tors}{\mathrm{tors}}
\nc{\cpl}{\mathrm{cpl}}
\nc{\loc}{\mathrm{loc}}

\nc{\doublefaktor}[3]{%
    {\textstyle #1}
    \mkern-4mu\scalebox{1.5}{$\diagdown$}\mkern-5mu^{\textstyle #2}%
    \mkern-4mu\scalebox{1.5}{$\diagup$}\mkern-5mu{\textstyle #3} }

\nc{\QW}{W^{\text{Q}}}

\nc{\mT}{\kern-0.5em\mod\kern-0.1em\text{-}\cat{T}^c}
\nc{\mTc}{\kern-0.5em\mod\kern-0.1em\text{-}\cat{T}^c}
\nc{\MTc}{\Mod\kern-0.1em\text{-}\cat{T}^c}
\nc{\MT}{\Mod\kern-0.1em\text{-}\cat{T}}
\newcounter{enum-resume-hack}
\Crefname{Thm}{Theorem}{Theorems}
\Crefname{Prop}{Proposition}{Propositions}
\Crefname{Lem}{Lemma}{Lemmas}
\Crefname{thmx}{Theorem}{Theorems}

\begin{document}
\setlength{\parindent}{0cm}
\setlength{\parskip}{0.8ex}
\title[The Galois theory of $G$-spectra]{The Galois theory of $G$-spectra \\ and the Burnside ring}

\author[Naumann]{Niko Naumann}
\author[Pol]{Luca Pol}
\author[Ramzi]{Maxime Ramzi}
\date{\today}

\address{Niko Naumann, Fakult{\"a}t f{\"u}r Mathematik, Universit{\"a}t Regensburg, Universit{\"a}tsstraße 31, 93053 Regensburg, Germany}
\email{Niko.Naumann@mathematik.uni-regensburg.de}
\urladdr{https://homepages.uni-regensburg.de/$\sim$nan25776/}

\address{Luca Pol, Max Planck Institute for Mathematics, Vivatsgasse 7, 53111 Bonn, Germany}
\email{pol@mpim-bonn.mpg.de}
\urladdr{https://sites.google.com/view/lucapol/}

\address{Maxime Ramzi,
FachBereich Mathematik und Informatik, Universit{\"a}t M{\"u}nster, Einsteinstraße 62 Germany}
\email{mramzi@uni-muenster.de}
\urladdr{https://sites.google.com/view/maxime-ramzi-en/home}

\begin{abstract}
We prove that the Galois groupoid of the category of $G$-spectra for a finite group $G$ is algebraic, i.e. equivalent to the \'etale fundamental groupoid of the Burnside ring of $G$. We implement an algorithm that computes the latter from the table of marks of $G$, and provide numerous examples.
\end{abstract}


\maketitle
\setcounter{tocdepth}{1}
\tableofcontents

\section{Introduction}
One of the most basic invariants of a topological space is its fundamental group. 
While originally defined in terms of maps into the space, it can equivalently be 
recovered from its category of sheaves by identifying covering spaces with locally 
constant sheaves. From this perspective, the van Kampen theorem becomes particularly transparent. More importantly, this viewpoint enabled the Grothendieck school to attach a fundamental profinite groupoid $\pi_{\leq 1}^{\et}(X)$ to any scheme $X$, despite the absence of a reasonable underlying topology, by means of the category of 
\'etale sheaves on $X$, see \cite{SGA1},\cite{SGA1-2}.
When $X$ is of finite type over $\mathbb{C}$, the resulting invariant agrees with 
the classical fundamental groupoid of $X(\mathbb{C})$ after profinite completion, 
by the Riemann existence theorem.
\subsection*{Axiomatic Galois theory}
More recently, Rognes \cite{Rognes2008} and Mathew \cite{Mathew2016} isolated the categorical properties of the category of \'etale sheaves that allow one to associate to any presentably symmetric monoidal stable $\infty$-category $\cat C$ a profinite fundamental groupoid $\pi_{\leq 1}(\cat C)$, often referred to as the \emph{Galois groupoid}
of $\cat C$. This leads to a far-reaching generalization of classical Galois theory,
with strong descent properties for $\pi_{\leq 1}(-)$ that significantly extend the
van Kampen theorem. In particular, this framework recovers the examples above and
produces many new ones. These invariants admit a comparison with classical \'etale fundamental groupoids via decategorification. For any presentably symmetric
monoidal stable $\infty$-category $\cat C$, there is a natural map
\[ 
\pi_{\leq 1}(\cat C)\to \pi_{\leq 1}^{\et}(\pi_0(\End_\cat C(\unit_\cat C)))
\]
to the \'etale fundamental groupoid of the commutative ring obtained from $\cat C$
by passing to $\pi_0$ of the endomorphisms of the unit. If this map is an equivalence,
we say that the Galois theory of $\cat C$ is \emph{algebraic}.

\subsection*{Equivariant stable homotopy theory}
This abstraction unlocks interesting new examples, such as the $\infty$-category of spectra $\Sp$.  Mathew, building on work of Rognes, showed that the Galois theory of spectra is algebraic, and hence trivial:
\[
\pi_{\leq 1}^{\et}(\pi_0(\End_{\Sp}(\unit_\mathcal \Sp)))=\pi_{\leq 1}^{\et}(\Z)=\ast.
\]
The last equality follows from a classical consequence of Minkowski's theorem, namely
that $\mathbb{Z}$ admits no nontrivial finite \'etale extensions. One may interpret
this result as saying that the sphere spectrum $\Sphere$ is Galois-closed: any Galois cover is split. 

The situation is richer in the equivariant setting. Let $\Sp_G$ denote the
$\infty$-category of genuine $G$-spectra for a finite group $G$. In this case, a
result of Segal \cite{Segal_burnside} identifies the endomorphism ring of the unit as
\[
\pi_0(\End_{\Sp_G}(\unit_{\Sp_G}))=A(G)
\]
the Burnside ring of $G$. Our first main result shows that Mathew's algebraicity
theorem extends to this setting.

\begin{Thm*}[\ref{cor-algebraic}]
      For every finite group $G$, the comparison map 
      \[ 
      \pi_{\leq 1}(\Sp_G)\xrightarrow{\sim} \pi_{\leq 1}^{\et}(A(G))
      \]
      is an isomorphism.
\end{Thm*}
To the best of our knowledge, the \'etale fundamental groupoid of $A(G)$ has not
previously been computed. In our previous paper \cite{Separablepaper} we verified that for a $p$-group $G$, the $G$-sphere is Galois-closed. 
However, it came as a surprise to the authors that for a cyclic group of order six, it is not. This motivates the following detailed study of the
\'etale fundamental groupoid of the Burnside ring.

\subsection{The \'etale fundamental groupoid of the Burnside ring}
The starting point of our calculation is the familiar injective mark homomorphism
\[ A(G)\hookrightarrow\prod_{(H)}\mathbb Z.\]
Since this is faithfully flat, it allows to access $\pi_{\leq 1}^{\et}(A(G))$ via glueing the simply connected copies of $\Spec(\mathbb Z)$.
To simplify notation, we write 
\[
\mathfrak F(n):=B\widehat{F(n)}
\]
for the groupoid associated to the profinite completion of the free group $F(n)$ on $n\ge 0$ generators. Our first structural result identifies the \'etale fundamental groupoid of the Burnside ring with the profinite completion of a disjoint union of bouquets of circles.

\begin{Thm*}[\ref{thm:structure}]\label{thm:structure}
Let $G$ be a finite group. Then 
\[ \pi^{\et}_{\leq 1}(A(G))\simeq \bigsqcup_{i=1}^{n(G)}\mathfrak F(\alpha_i)\]
for uniquely determined integers $n(G)\ge 1$ and $\alpha_i\ge 0$.
\end{Thm*}

We encode these invariants in an unordered list
\[
L = [\alpha_1,\dots,\alpha_{n(G)}].
\]
The proof of the above theorem leads to an explicit algorithm computing $L$ from the 
table of marks of $G$, which we implement in GAP. We highlight a few consequences:

\begin{itemize}
    \item If $G$ is solvable, then $L=[N]$, where
    \[
    N =
    \sum_{(K)\subsetneq (G)}
    \Bigl(
    |\{ p \text{ prime} : p \mid |W_G(K)|\}| - 1
    \Bigr),
    \]
    and the sum runs over conjugacy classes of proper subgroups $(K)$ and $W_G(K)$ denotes the Weyl group, see \cref{thm:solvable}.    
     \item If $G$ is a $p$-group, then the above formula implies that $\pi_{\leq 1}^{\et}(A(G))$ is trivial, i.e. 
    $L=[0]$. In this case the $G$-sphere spectrum is Galois-closed recovering \cite[Corollary 9.10]{Separablepaper}.
    \item If $G=\mathbb{Z}/p^{n+1}q$ for distinct primes $p \neq q$, then the above formula simplifies to $L=[n]$. 
    In particular, every finitely generated free profinite group arises in this way.
    
    \item We compute further examples in GAP, including large simple groups such as the 
    Mathieu group $M_{24}$; see \cref{tab:group_L}.
\end{itemize}
\subsection{Separable algebras vs Galois covers}
In \cite{Separablepaper}, we described, for a $p$-group $G$, all separable dualizable algebras in $G$-spectra: they are all ``standard'', that is, of the form $\Sphere_G^X$ for some finite $G$-set $X$. We also proved that for the smallest non-$p$-group (namely $C_6$, or more generally $C_{pq}$ for distinct primes $p,q$), there are some non-standard separable dualizable algebras. The presence of the latter seem to make a full calculation of the ``finite \'etale site'' of $\Sp_G$ quite intricate, and at present untractable. The present paper stems from the observation that if one restricts to Galois covers, the difficulties in question go away, and we are able to give a satisfactory description far beyond the $p$-group case.
\subsection{Outline}

In \Cref{sec:axiomatic}, we review Mathew's approach to Galois theory and collect 
some general results needed later.

In \Cref{sec:equivariant}, we recall the equivariant tools developed in our previous 
work, in particular a system of fracture squares refining isotropy separation, which 
allows us to study $\Sp_G$ and its Galois groupoid via descent.

In \Cref{sec:main}, we prove the algebraicity theorem by induction on families of 
subgroups. A key feature is that, although the fracture squares involve Tate 
constructions whose Galois theory is not well understood, their contribution to the 
gluing can still be controlled; this ultimately relies on the Segal conjecture.

In \Cref{sec:burnside}, we analyze the \'etale fundamental groupoid of $A(G)$.  In fact, our methods determine the 
entire \'etale homotopy type of $\Spec(A(G))$ but we leave the details to the interested reader.

Finally, we complement these structural results with an explicit algorithm computing 
$\pi_{\leq 1}^{\et}(A(G))$ from the table of marks, together with a collection of 
examples and accompanying GAP code.

\subsection{Acknowledgment}
We thank Akhil Mathew for pointing out the relevance of the Segal conjecture, and Scott Balchin for helping with the diagrams in \cref{sec-diagram}. We thank the anonymous referee of our paper \cite{Separablepaper} for suggesting we look at the Galois theory.

\section{Axiomatic Galois theory}
\label{sec:axiomatic}
In this section we recall the basic background on Galois theory which we will need for this paper. The main reference is \cite{Mathew2016}.

\begin{Rec}\label{rec-Galoiscorrespondence}
Following ideas of Grothendieck, Mathew \cite[Definition 5.15]{Mathew2016}, introduced the notion of a Galois category. The prototypical example of a Galois category is the functor category $\Fun(\mathcal{G}, \Fin)$ from a finite groupoid $\mathcal{G}$ to the category of finite sets. If $\mathcal{G}=BG$ for a finite group $G$, then we recover the usual category $\Fin_G$ of finite $G$-sets as an example of a Galois category. The collection of Galois categories can be organized into
a $(2,1)$-category $\mathsf{GalCat}$. The previous construction defines a functor 
\[
\mathsf{Grp}_{\fin}^{\mathrm{op}} \to \mathsf{GalCat}, \quad \mathcal{G} \mapsto \Fun(\mathcal{G}, \Fin)
\]
from the $(2,1)$-category of finite groupoids. Since the category of Galois category admits filtered colimits \cite[Proposition 5.34]{Mathew2016} the above functor extends uniquely to a functor 
\begin{equation}\label{Galois correspondence}
\mathrm{Pro}(\mathsf{Grp}_{\fin})^{\mathrm{op}} \simeq \Ind(\mathsf{Grp}_{\fin}^{\mathrm{op}}) \to\mathsf{GalCat}
\end{equation}
from the $(2,1)$-category of profinite groupoids. The Galois correspondence \cite[Theorem 5.36]{Mathew2016} states that the functor \eqref{Galois correspondence} is an equivalence. This gives us a way to access Galois categories via profinite groupoids, and viceversa. 
\end{Rec}
\begin{Not}\label{warning}
We write 
\[
    \begin{tikzcd}
     \mathcal G \arrow[r, two heads,"f"]  & \mathcal{G'} 
        \end{tikzcd}
  \]  
and refer to as ``surjective'', a map of profinite groupoids which corresponds, by the Galois correspondence, to a fully faithful functor of Galois categories. One can check that this is equivalent to $\pi_0(f)$ being bijective and $\pi_1(f,x)$ surjective for all $x\in\pi_0(\mathcal G)$. We omit the proof as we do not need to use this result.
\end{Not}

\begin{Lem}\label{lem-profinite}
    Every square of profinite groupoids of the form 
    \[
    \begin{tikzcd}
        \mathcal G \arrow[r, two heads,"f"] \arrow[d] & \mathcal{G'} \arrow[d] \\
        \ast \arrow[r] & \ast
    \end{tikzcd}
    \]
    is a pushout. 
\end{Lem}
\begin{proof}
    By definition, we have to verify that the corresponding square of Galois categories 
    \[
     \begin{tikzcd}
        \Fin \arrow[r] \arrow[d] & \Fin \arrow[d] \\
        \cat C_{\mathcal{G'}} \arrow[r, hook, "f^*"] & \cat C_{\mathcal{G}}
    \end{tikzcd}
    \]
    is a pullback. Write $P$ for this pullback. By \cite[Lemma 5.37]{Mathew2016}, finite limits of Galois categories are calculated in categories, and this implies that the functor $P \hookrightarrow \Fin$ is fully faithful. Using that $\Fin$ is the initial Galois category, we find a functor $\Fin \to P$ which is the identity after composing with $P \hookrightarrow \Fin$. This shows that $P = \Fin$ as required. 
\end{proof}

There is a functorial way to associate a Galois category with every presentably symmetric monoidal stable $\infty$-category $\cat C \in \CAlg(\PrL)$: 

\begin{Def}\label{def-finite-cover}
    An object $A \in \CAlg(\cat C)$ is a \emph{finite cover} if there exists $A' \in \CAlg(\cat C)$ such that:
    \begin{enumerate}
        \item $A'$ is descendable: the thick tensor ideal generated by $A'$ is all of $\cat C$.
        \item  there is an equivalence 
 \[
 A \otimes A' \simeq \prod_{i=1}^n A'[e_{i}^{-1}] \qquad \mathrm{in} \;\;\CAlg(\Mod_{A'}(\cat C))
 \]
 where each $e_i$ is an idempotent of $A'$.
    \end{enumerate}
    The finite covers span a full subcategory $\CAlg^\cov(\cat C)\subseteq \CAlg(\cat C)$. 
\end{Def}
We will see later on that the finite covers give rise to a Galois category, and so can associated a profinite groupoid to it by the Galois correspondence. Before discuss this, we record the following nontrivial fact which we will need later in the paper.
\begin{Prop}\label{prop-descendable}
    Let $\cat C$ be a presentably symmetric monoidal stable $\infty$-category which is connected, i.e. the commutative ring $\pi_0(\unit)\coloneqq \pi_0(\End_{\cat C}(\unit))$ is indecomposable. Then any nonzero finite cover in $\cat C$ is descendable. 
\end{Prop}
\begin{proof}
    By \cite[Theorem A]{NaumannPol} a finite cover $A$ of $\cat C$ can equivalently be described as a separable commutative algebra with dualizable underlying object, with associated tt-degree function 
    \[
    \deg(A) \colon \Spc(\cat C^{\dual}) \to \Z
    \]
    which is locally constant.  Since $\cat C$ is connected, the Balmer spectrum $\Spc(\cat C^{\dual})$ is connected and non-empty by \cite[Proposition 3.7]{NaumannPol} and so the tt-degree function of $A$ must be constant. If $A$ is nonzero, then $\deg (A)$ is nonzero, too. Finally, we apply \cite[Lemma 5.19]{NaumannPol} to conclude.
\end{proof}
\begin{Lem}\label{lem-galois-fincov}
    Let $\cat C$ be a presentably symmetric monoidal stable $\infty$-category. Then $\CAlg^{\cov}(\cat C)\op$ is a Galois category. Moreover, every finite cover in $\cat C$ has underlying object that is dualizable in $\cat C$. 
    Finally, if $F \colon \cat C \to \cat D$ is a map in $\CAlg(\PrL)$ that induces an equivalence on dualizable objects, then it induces an equivalence on finite covers. 
\end{Lem}
\begin{proof}
    The first claim is proved in \cite[Theorem 6.5]{Mathew2016}. The proof of \cite[Theorem 6.5]{Mathew2016} shows that any finite cover $A$ has dualizable underlying object, and that the descendable algebra $A'$ as in \cref{def-finite-cover} can be chosen to be dualizable too. Since descendability of such $A'$ can be checked in the subcategory of dualizable objects, e.g. by \cite[Lemma 7.15]{descent}, we obtain the final claim as well.
\end{proof}

\begin{Def}\label{def-Galois-group}
    The \emph{Galois groupoid} of $\cat C$ is the profinite groupoid $\pi_{\leq 1}(\cat C)$ associated via the Galois correspondence to $\CAlg^{\cov}(\cat C)\op$. If $\cat C$ is connected, the Galois groupoid can be represented by a profinite group. 
\end{Def}

\begin{Rec}\label{rec-algebraic}
   Given a presentably symmetric monoidal stable category $\cat C$, we can consider the associated category 
   \[
   \mathrm{Cov}_{\pi_0(\unit)}
   \]
   of finite \'etale algebras over $\pi_0(\unit)$. By \cite[Proposition 6.10]{Mathew2016}, there is a fully faithful functor of Galois categories
   \[
   \mathrm{Cov}_{\pi_0(\unit)}\subseteq \CAlg^{\cov}(\cat C)\op,
   \]
  classified by a surjective map of profinite groupoids
   \begin{equation}\label{alg-map}
   \pi_{\leq 1}(\cat C) \twoheadrightarrow \pi_{\leq 1}^\alg(\cat C)
   \end{equation}
   which is natural in $\cat C$. 
   Here, $\pi_{\leq 1}^\alg(\cat C)$ denotes the \'etale fundamental groupoid of the ring $\pi_0(\unit)$ considered in \cite{SGA1}, which we also denote as $\pi_{\leq 1}^{\et}(\pi_0(\unit))$. 
\end{Rec}
\begin{Ter}
    We say that the Galois groupoid of $\cat C$ is \emph{algebraic} if the map \eqref{alg-map} is an equivalence. 
\end{Ter}
\begin{Exa}\label{ex-Galois-R}
    Let $R$ be a discrete commutative ring and let $\Mod(R)$ denote the $\infty$-category of $R$-module in spectra. Then 
    \[
    \pi_{\leq 1}(\Mod(R))\simeq \pi_{\leq 1}^{\alg}(\Mod(R))=\pi_{\leq 1}^{\et}(R)
    \]
    by \cite[Theorem 6.17]{Mathew2016}.
\end{Exa}
Later on in this paper, it will be convenient to work with $2$-rings $\CAlg(\mathrm{Cat}_{\infty}^{\mathrm{perf}})$ rather than with their large variant $\CAlg(\PrL)$.
\begin{Def}\label{rem-galois-dual}
     Given a $2$-ring $\cat R$ we define its Galois groupoid $\pi_{\leq 1}(\cat R)$ (resp. algebraic Galois groupoid $\pi_{\leq 1}^{\alg}(\cat R)$) as the Galois groupoid (resp. algebraic Galois groupoid) associated to $\Ind(\cat R)$. Note that for any $\cat C \in \CAlg(\PrL)$, the canonical map $\Ind(\cat C^{\dual}) \to \cat C$ induces an isomorphism on Galois groupoids by \cref{lem-galois-fincov}. 
\end{Def}

\section{Equivariant tools}
\label{sec:equivariant}
In this section, we recall the stratification of genuine $G$-spectra required for our computation of the Galois groupoid, following \cite{Glueingpaper}. We refer the reader there for further details.
\begin{Rec}
    For a finite group $G$, we let $\Sp_G$ denote the $\infty$-category of genuine $G$-spectra. There are several equivalent constructions of this $\infty$-category: 
    \begin{itemize}
        \item as the underlying $\infty$-category associated with the category of orthogonal $G$-spectra with the stable model structure of \cite{Mandell-May}.
        \item from the $\infty$-category of pointed $G$-spaces $\cat S_{G,\ast}$ by inverting representation spheres, see \cite[Definition C.1]{Gepner-Meier}.
        \item as spectral Mackey functors, see \cite[Appendix A]{CANN2024}.
    \end{itemize}
    The $\infty$-category $\Sp_G$ is a presentably symmetric monoidal stable $\infty$-category that comes equipped with a symmetric monoidal colimit-preserving suspension functor $ \cat S_{G,\ast} \to \Sp_G$. In fact, the $\infty$-category $\Sp_G$ admits a set of compact and dualizable generators $G/H_+$ for $H \subseteq G$ (we omit writing the suspension spectrum functor).  For more background on equivariant stable homotopy theory we refer the reader to \cite{LMS86}, \cite[Appendices A and B]{HHR} and \cite{MNN}.
\end{Rec}
\begin{Not}
Let $G$ be a finite group and $\cat F$ be a family of subgroups of $G$. We write $\Sp_G/ \cat F$ for the smashing localization of $\Sp_G$ which annihilates all the orbits $G/H_+$ with $H \in \cat F$. By definition there is a smashing localization functor 
\[
L_{\cat F} \colon \Sp_G \to \Sp_G/\cat F
\]
whose idempotent algebra is given by $\widetilde{E\cat F}$.
\end{Not}

\begin{Ter}\label{ter-minimally-not-F}
     Let $G$ be a finite group and $\cat F$ be a family of subgroups of $G$. We say that a subgroup $K\subseteq G$ is \emph{minimally} not in $\cat F$ if $K \not \in \cat F$, but any proper subgroup of $K$ belong to $\cat F$. The smallest family $\cat F\cup K$ containing $\cat F$ and $K$ adds to $\cat F$ all $G$-conjugates of $K$. In the literature, the families $\cat F$ and $\cat F\cup K$ are referred to as \emph{adjacent}.
\end{Ter}
 The key ingredient that we will need for our computation of the Galois groupoid of $G$-spectra is the following:
 
\begin{Prop}\label{prop-key-pullback}
    Let $\cat F$ be a family of subgroups of $G$ and let $K\subseteq G$ be a subgroup which is minimally not in $\cat F$. Then there is a pullback of $2$-rings
    \[
    \begin{tikzcd}
        (\Sp_G/\cat F)^\omega \arrow[r, "L_{\cat F\cup K}"]\arrow[d, "\Phi^K"'] & (\Sp_G/\cat F \cup K)^\omega\arrow[d,"\widetilde{\Phi}^K"] \\
        \Perf(\Sphere)^{hW_G(K)} \arrow[r] & \Perf(\Sphere)^{tW_G(K)}
    \end{tikzcd}
    \]
    where $\Phi^K$ denotes the $K$-geometric fixed points functor. 
\end{Prop}
\begin{proof}
    See \cite[Corollary 7.12]{Glueingpaper} where we can identify the bottom right corner with the Tate construction by \cite[Example 4.14]{Separablepaper}.
\end{proof}
We next recall some background on the Burnside ring. Standard references are \cite{Tomdieck} and \cite{Bouc}.
\begin{Rec}\label{rec-Burnside}
    For a finite group $G$, we let $A(G)$ denote the Burnside ring of $G$, i.e. the Grothendieck ring of finite $G$-sets. The addition (resp. multiplication) in $A(G)$ is induced by disjoint union of $G$-sets (resp. their cartesian product). It is easy to see that $A(G)$ is a free $\Z$-module with basis given by the cosets $[G/H_1],\ldots ,[G/H_s]$ for a choice $H_1,\ldots, H_s$ of representatives of the conjugacy classes of subgroups of $G$. There is an augmentation ring homomorphism 
    \[
    \epsilon\colon A(G) \longrightarrow \Z\,  \qquad \sum_i a_i [G/H_i] \mapsto \sum_i a_i\cdot |G/H_i|
    \]
    with augmentation ideal denoted by $I(G):=\mathrm{ker}(\epsilon)$ or simply $I$. 
    If $\cat F$ is a family of subgroups of $G$, we write $A(G)/\cat F$ for the quotient of $A(G)$ by the ideal generated by all $[G/H]$ with $H\in\cat F$.
\end{Rec}

\begin{Prop}\label{prop-pullback-unit}
    Passing to the endomorphism rings of the unit objects in \cref{prop-key-pullback}, we obtain the following pullback of connective $\bbE_\infty$-rings
     \[
    \begin{tikzcd}
        (\widetilde{E\cat F})^G\arrow[r]\arrow[d] & (\widetilde{E\cat F \cup K})^G\arrow[d] \\
        \Sphere^{hW_G(K)} \arrow[r] & \Sphere^{tW_G(K)}.
    \end{tikzcd}
    \]
    Applying $\pi_0$ we obtain the following commutative square of commutative rings 
    \begin{equation}\label{importantsquare}
    \begin{tikzcd}
        A(G)/\cat F\arrow[r,"q_K"]\arrow[d,"(-)^K"'] & A(G)/(\cat F \cup K)\arrow[d,"\widetilde{(-)}^K"] \\
        A(W_G(K))^\wedge_{I}\arrow[r,"\mathrm{quot}"] & A(W_G(K))^\wedge_I/[W_G(K)].
    \end{tikzcd}
    \end{equation}
    where 
    \begin{itemize}
    \item $q_K$ factors out the principal ideal generated by $[G/K]$, 
    \item the vertical arrows are induced by the $K$-fixed points via $[X] \mapsto [X^K]$, 
    \item $\mathrm{quot}$ is the quotient map. 
    \end{itemize}
\end{Prop}
\begin{proof}
    Let us first identify the resulting square of $\mathbb{E}_\infty$-rings. Since $\Sp_G/\cat F$ is a smashing localization of $\Sp_G$ with idempotent algebra given by $\widetilde{E \cat F}$, there is a canonical equivalence 
    \[
    \Sp_G/\cat F \simeq \Mod_{\Sp_G}(\widetilde{E\cat F})
    \]
    see \cite[Proposition 4.8.2.10]{HA}. From this we immediately see that the endomorphism ring of the unit object in $(\Sp_G/\cat F)^\omega$ is given by $(\widetilde{E\cat F})^G$.
    The same argument applies to $(\Sp_G/\cat F \cup K)^\omega$ giving as a result $(\widetilde{E\cat F \cup K})^G$. For the $\mathbb{E}_\infty$-ring appearing in the bottom left corner, we use that mapping spaces commutes with limits to obtain $\Sphere^{hW_G(K)}$. Finally for the bottom right corner we use that the Weyl group action on $\Sphere$ is trivial together with \cite[Remark 4.8 and Example 4.14]{Separablepaper}. Note that  the top $\mathbb{E}_\infty$-rings are connective by the tom Dieck splitting, and the bottom ones are as a consequence of the Segal conjecture.

    We now argue that by applying $\pi_0$ we obtain the claimed square of commutative rings. In the proof of \cite[Theorem 5.5]{Separablepaper} we argued that $\pi_0(\Sphere^{hW_G(K)})=A(W_G(K))^\wedge_I$ and $\pi_0(\Sphere^{tW_G(K)})=A(W_G(K))^\wedge_I/[W_G(K)]$ and that the corresponding map is the quotient map $\mathrm{quot}$. The isotropy separation sequence
      \[
    E\cat F_+ \to \Sphere_G \to  \widetilde{E\cat F}
    \]
    gives rise to an exact sequence 
    \[
    \pi_0^G(E\cat F_+) \to A(G) \to \pi_0^G(\widetilde{E\cat F}) \to 0.
    \]
By \cite[Appendix A.1]{MNN2019} we can identify $E\cat F$ with the colimit of $G/H$ over the $\cat F$-orbit category $\OO_{\cat F}(G)$, and the map $E\cat F_+\to \Sphere_G$ is induced by the transfer maps $G/H_+\to \Sphere_G$. It follows from  \cite[V. Proposition 9.3]{LMS86} that the image of this transfer map in the Burnside ring $A(G)\cong \pi_0^G(\Sphere_G)$ is exactly the ideal generated (additively) by the classes of $G/H$ with $H\in\cat F$. Altogether, this shows that $\pi_0^G(\widetilde{E\cat F})=A(G)/\cat F$, and similarly for $\cat F \cup K$. Moreover the corresponding map $(\widetilde{E\cat F})^G\to (\widetilde{E\cat F \cup K})^G$ is induced by the containment $\cat F \subseteq \cat F \cup K$, and so has kernel generated by the orbit $G/K$. The effect on $\pi_0$ is precisely described by the map $q_K$. The vertical arrows are induced by the $K$-geometric fixed points and so the resulting effect on $\pi_0$ is just given by taking usual $K$-fixed points of $G$-sets, recalling that geometric fixed points of suspension spectra are the suspension spectra of categorical fixed points. 
\end{proof}

\begin{Rem}\label{rem-connected}
    For every finite group $G$, the category $\Perf(\Sphere)^{hG}$ is connected, i.e., $\pi_0(\unit)$ is indecomposable: By \cref{prop-pullback-unit}, we can identify $\pi_0(\unit)$ with the $I$-adically complete ring $\widehat{A(G)}_I$ which is a Henselian ring. By the lifting of idempotents for Henselian rings \cite[Tag 09XI]{stacks-project}, we can check the indecomposability of $\widehat{A(G)}_I$ after passing to the quotient $\widehat{A(G)}_I/I \simeq \Z$, from which the claim follows.
\end{Rem}

\section{The proof of algebraicity}
\label{sec:main}
In this section we prove that the Galois groupoid of $G$-spectra is algebraic. 

\begin{Prop}\label{prop-Galois-perfShG}
    For every finite group $G$, the Galois groupoid of $ \Perf(\Sphere)^{hG}$ is trivial.
\end{Prop}
\begin{proof}
   By  \cite[Proposition 7.2]{Mathew2016} we have fully faithful functor 
   \[
   \CAlg^{\cov}(\Perf(\Sphere)^{hG}) \hookrightarrow \CAlg^{\cov}(\Perf(\Sphere))^{hG}=\Fin_G
   \]
   using that the Galois groupoid of the sphere is trivial by \cite[Example 6.19]{Mathew2016}. Therefore, any finite cover in $\Perf(\Sphere))^{hG}$ is of the form $\Sphere^X$ for some finite $G$-set $X$. 
   We noted in \cref{rem-connected} that this Galois theory is connected. Therefore by \cref{prop-descendable} the algebra $\Sphere^X$ is a finite cover if and only if it is descendable. By decomposing $X$, it now suffices to see that if $\Sphere^{G/H}$ is descendable, then $H=G$. 
    We verify this after applying the fully faithful symmetric monoidal functor (\cite[Corollary 6.2]{MNN}):
    \[
    \Perf(\Sphere)^{hG} \hookrightarrow \Mod_{\underline{\Sphere}}(\Sp_G)^\omega
    \]
    where $\underline{\Sphere}$ denotes the Borel completed sphere. Recall from \cite[Theorem 4.25]{MNN2019} that the derived defect base of the Borel completed sphere is all $p$-groups for all primes $p$. So, if $\Sphere^{G/H}$ is descendable, then the smallest family containing $H$ contains all $p$-subgroups of $G$ (for all primes $p$). In particular, every $p$-Sylow subgroup is sub-conjugate to $H$, and so by cardinality $G=H$. 
\end{proof}

\begin{Prop}\label{cor-pushout-galois}
 Let $\cat F$ be a family of subgroups of $G$ and let $K\subseteq G$ be a subgroup which is minimally not in $\cat F$. Then the square from \cref{prop-key-pullback} induces a pushout of profinite groupoids 
\[
\begin{tikzcd}
\pi_{\leq 1}(\Perf(\Sphere)^{tW_G(K)}) \arrow[r] \arrow[d] & \pi_{\leq 1}(\Sp_G/\cat F \cup K) \arrow[d] \\
* \arrow[r] & \pi_{\leq 1}(\Sp_G/\cat F).
\end{tikzcd}
\]
\end{Prop}
\begin{proof}
    By \cite[Proposition 7.2]{Mathew2016} the formation of finite covers commutes with finite limits of $2$-rings. Therefore if we apply the functor of finite covers to the pullback in \cref{prop-key-pullback} we still obtain a pullback, and then using the Galois correspondence to pass to the Galois groupoids, we obtain a pushout. The indicated corner is trivial by \cref{prop-Galois-perfShG}.
\end{proof}

We now want to obtain a similar pushout for the algebraic Galois groupoids. To this end we recall the following important feature of the \'etale fundamental groupoid. 
\begin{Rem}\label{rem-henselian}
    For any henselian pair $(R,I)$, the reduction $R\to R/I$ induces an isomorphism on \'etale fundamental groupoids \cite[Lemma 09ZL]{stacks-project}. In particular, this applies to any $I$-adically complete Noetherian ring $R$.
\end{Rem}

\begin{Lem}\label{lem-alg-galois}
    Let $G$ be a finite group. Then 
    \[
    \pi_{\leq 1}^\alg(\Perf(\Sphere)^{hG})=*\quad \mathrm{and}\quad \pi_{\leq 1}^\alg(\Perf(\Sphere)^{tG})\simeq \pi_{\leq 1}^{\et}(\Z/|G|).
    \]
\end{Lem}
\begin{proof}
    The first claim follows from \cref{rec-algebraic} and the fact that the Galois groupoid is trivial by \cref{prop-Galois-perfShG}. For the second claim we use \cref{prop-pullback-unit} to identify the algebraic Galois groupoid with the \'etale fundamental groupoid of $A(G)^\wedge_I/[G]$. 
    Now note that the quotient ring $A(G)^\wedge_I/([G]))$ is a finitely generated module over the complete Noetherian ring $A(G)^\wedge_I$, hence $I$-adically complete. We are only left to apply \cref{rem-henselian} and observe that 
    \[
    A(G)^\wedge_I/([G],I)=\Z/|G|.
    \]
\end{proof}
\begin{Prop}\label{cor-pullback-on-perf}
     Let $\cat F$ be a family of subgroups of $G$ and let $K\subseteq G$ be a subgroup which is minimally not in $\cat F$. Passing to the endomorphism rings of the unit object in \cref{prop-key-pullback} and taking categories of perfect objects yields a pullback of $2$-rings
    \begin{equation}\label{secondsquare}
    \begin{tikzcd}
        \Perf((\widetilde{E\cat F})^G)\arrow[r]\arrow[d] & \Perf((\widetilde{E\cat F \cup K})^G)\arrow[d] \\
        \Perf(\Sphere^{hW_G(K)} )\arrow[r] & \Perf(\Sphere^{tW_G(K)}).
    \end{tikzcd}
    \end{equation}
    Moreover, for these categories the comparison map 
    $\pi_{\leq 1}\to\pi_{\leq 1}^\alg$ is an isomorphism.
\end{Prop}
\begin{proof}
    We start by noting that the bottom horizontal map in \cref{prop-pullback-unit} induces a surjection on $\pi_0$. If we apply the functor $\Mod(-)$ to this latter square, we obtain a pullback of categories once we restrict to the subcategories of connective modules, see \cite[Proposition 16.2.2.1]{SAG}. In fact, the same proof holds more generally for bounded below objects, and so it applies to perfect objects as well (noting that being perfect in a pullback is detected vertex-wise), giving the first claim. We have already noted in \cref{prop-pullback-unit} that all the $\mathbb{E}_\infty$-rings in the square are connective so by \cite[Theorem 6.17]{Mathew2016} the indicated comparison maps are all isomorphisms. 
\end{proof}
We now want to prove the following result.

\begin{Prop}\label{prop-alg-galois-pullback}
Applying $\pi^{\alg}_{\leq 1}$ to the square \eqref{secondsquare} we obtain a pushout of profinite groupoids 
\[
\begin{tikzcd}
 \pi_{\leq 1}^{\alg}(\Perf(\Sphere)^{tW_G(K)})  \arrow[r] \arrow[d] & \pi_{\leq 1}^\alg(\Sp_G/\cat F \cup K) \arrow[d] \\
* \arrow[r] & \pi_{\leq 1}^\alg(\Sp_G/\cat F)
\end{tikzcd}
\]
In particular, the \'etale fundamental groupoid functor sends the square \eqref{importantsquare} to a pushout. 
\end{Prop}

\begin{proof}
    By definition, applying $\pi_{\leq 1}^{\alg}$ to the square \eqref{secondsquare} is the same as applying the \'etale fundamental groupoid to \eqref{importantsquare}, see \cref{rec-algebraic}, so the second claim follows from the first one. Now note that by \cref{prop-alg-galois-pullback} there is no difference between $\pi_{\leq 1}^{\alg}$ and $\pi_{\leq 1}$ so we can instead argue using $\pi_{\leq 1}$. Recall that the Galois groupoid of Mathew sends pullbacks of $2$-rings to pushout of profinite groupoid by the Galois correspondence and \cite[Proposition 7.2]{Mathew2016}. Thus if we apply $\pi_{\leq 1}$ to the square in \cref{cor-pullback-on-perf} we obtain a pushout, and we can identify the terms using \cref{lem-alg-galois}.
\end{proof}

We now prove the induction step for our main result.

\begin{Prop}\label{prop-induction}
    Let $\cat F$ be a family of subgroups of $G$ and let $K\subseteq G$ be a subgroup which is minimally not in $\cat F$. Assume that the Galois groupoid of $\Sp_G/\cat F \cup K$ is algebraic. Then the Galois groupoid of $\Sp_G/\cat F$ is also algebraic.
\end{Prop}
\begin{proof}
    The core of the proof is the following cube:
    \[\resizebox{\columnwidth}{!}{$\displaystyle
\begin{tikzcd}[ampersand replacement=\&]
	\& {\pi_{\leq 1}^{\alg}(\Perf(\Sphere)^{tW_G(K)}) } \&\&\& {\pi_{\leq 1}^{\alg}(\Sp_G/\cat F \cup K)} \\
	{\pi_{\leq 1}(\Perf(\Sphere)^{tW_G(K)}) } \&\& {\pi_{\leq 1}(\Sp_G/\cat F \cup K)} \\
	\\
	\& * \&\&\& {\pi_{\leq 1}^{\alg}(\Sp_G/\cat F)} \\
	* \&\& {\pi_{\leq 1}(\Sp_G/\cat F)}
	\arrow[from=1-2, to=1-5]
	\arrow[from=1-2, to=4-2]
	\arrow[from=1-5, to=4-5]
	\arrow[two heads, from=2-1, to=1-2]
	\arrow[from=2-1, to=2-3]
	\arrow[from=2-1, to=5-1]
	\arrow[two heads, from=2-3, to=1-5]
	\arrow[from=2-3, to=5-3]
	\arrow[from=4-2, to=4-5]
	\arrow[two heads, from=5-1, to=4-2]
	\arrow[from=5-1, to=5-3]
	\arrow[two heads, from=5-3, to=4-5]
\end{tikzcd}
$}
\]
The front square is the pushout from \cref{cor-pushout-galois} and the back square is the pushout from \cref{prop-alg-galois-pullback}. The maps from the front to the back square are the canonical maps which relate the Galois groupoid to the algebraic one, so we know they are surjective and the cube commutes. By \cref{lem-profinite} the left square is a pushout. Using the pasting lemma for pushouts \cite[Proposition
4.4.2.1]{HTT} one deduces that the right most square is a pushout. Since we assumed that the top horizontal map of this square is an isomorphism, we deduce that also the bottom horizontal map is an isomorphism, concluding the proof. 
\end{proof}
\begin{Thm}
     Let $\cat F$ be a family of subgroups of $G$ . Then the Galois groupoid of $\Sp_G/\cat F$ is algebraic and so isomorphic to $\pi_{\leq 1}^{\et}(A(G)/\cat F)$.
\end{Thm}
\begin{proof}
     Choose an exhaustive filtration of families of subgroups of $G$:
    \[
    \cat F=\cat F_0 \subseteq \cat F_1 \subseteq \ldots \subseteq \cat F_{i-1} \subseteq \cat F_{i}\subseteq \ldots \subseteq \cat F_n=\mathrm{All}
    \]
   with each $\cat F_{i}\backslash \cat F_{i-1}= \{(K_i)\}$ consisting of a single conjugacy class. We prove by descending induction on $i\leq n$ that the Galois groupoid of $\Sp_G/\cat F_i$ is algebraic. The claim trivially holds for $\cat F_n$ since in this case $\Sp_G/\cat F_n =0$. The induction step holds by \cref{prop-induction} so we can conclude. We can finally identify the Galois group with the \'etale fundamental groupoid of $A(G)/\cat F$ using \cref{prop-pullback-unit}.
\end{proof}

The special case $\cat F=\emptyset$ of this result reads:

\begin{Cor}\label{cor-algebraic}
    For every finite group $G$, the Galois groupoid of $\Sp_G$ is isomorphic to the \'etale fundamental groupoid of the Burnside ring of $G$:
    \[ \pi_{\leq 1}(\Sp_G)\simeq \pi_{\leq 1}^{\et}(A(G)).\]
\end{Cor}

For the computations in the next section, we assemble the following.

\begin{Prop}\label{prop:explicit_po}
Let $G$ be a finite group, $\cat F$ a family of subgroups and $K\subseteq G$ a subgroup minimally not in $\cat F$. Then the square from \cref{prop-pullback-unit} can be expanded as follows:
\[
    \begin{tikzcd}[column sep=2cm]
        A(G)/\cat F\arrow[r,"q_K"]\arrow[ddd, bend right=90,"a"']\arrow[d,"(-)^K"'] & A(G)/(\cat F \cup K)\arrow[d,"\widetilde{(-)}^K"] \arrow[ddd, bend left =90,"b"]\\
        A(W_G(K))^\wedge_{I}\arrow[r,"\mathrm{quot}"]\arrow[d,"\pi_I"'] & A(W_G(K))^\wedge_I/[W_G(K)]\arrow[d,"\pi_I"]\\
        A(W_G(K))^\wedge_{I}/I \arrow[d, "\sim"]\arrow[r,"\mathrm{quot}/I"]& A(W_G(K))^\wedge_I/([W_G(K)],I)\arrow[d, "\sim"] \\
        \Z \arrow[r,"\mathrm{proj}"] & \Z/|W_G(K)|.
    \end{tikzcd}
    \]
where: 
\begin{itemize}
\item The map $\pi_I$ denotes the quotient map by the augmentation ideal $I$;
\item The map $a$ sends $[X]$ to $|X^K|$, and similarly for $b$.
\end{itemize}
Furthermore, the outer square induces a pushout of \'etale fundamental groupoids, and we have factorizations of $\pi^{\et}_{\leq 1}(a)$ and $\pi^{\et}_{\leq 1}(b)$ as follows:
\[
 \begin{tikzcd}
       \pi_{\leq 1}^{\et}(b) \colon\pi_{\leq 1}^{\et}(\mathbb Z/|W_G(K)|) \arrow[r]\arrow[d] & \pi_0(\pi_{\leq 1}^{\et}(\mathbb Z/|W_G(K)|))\arrow[r]\arrow[d] & \pi_{\leq 1}^{\et}(A(G)/\cat F\cup K)\arrow[d] \\
         \pi_{\leq 1}^{\et}(a) \colon\pi_{\leq 1}^{\et}(\mathbb Z) \arrow[r] & \pi_0(\pi_{\leq 1}^{\et}(\mathbb Z)) \arrow[r] & \pi_{\leq 1}^{\et}(A(G)/\cat F).
    \end{tikzcd}
\]
\end{Prop}

\begin{proof}
The first part of the proposition is just an expansion of the square \eqref{importantsquare}. Moreover, the second claim in \cref{prop-alg-galois-pullback} tell us that the top square is sent to a pushout by the \'etale fundamental groupoid, and \cref{rem-henselian} tells us that the reduction map $\pi_I$ induces an isomorphism between \'etale fundamental groupoids. All together we find that the outer square is sent to a pushout by the \'etale fundamental groupoid functor. 
Regarding the factorization, we recall that $\pi^{\et}_{\leq 1}(\mathbb Z)=*$ by Minkowski's theorem, so we only need to argue that $\pi_{\leq 1}^{\et}(b)$ factors as claimed.

We provide two proofs for this:
\begin{enumerate}
    \item[(1)] We first observe that 
\[ \prod_{0<p | |W_G(K)|}\pi^{\et}_{\leq 1}(\mathbb F_p)\stackrel{\simeq}{\longrightarrow} \pi^{\et}_{\leq 1}(\mathbb Z/|W_G(K)|)\]
is an isomorphism to reduce to factoring any given composition

\[
 \begin{tikzcd}
      b' \colon A(G)/\cat F\cup K \arrow[r, "b"] & \mathbb Z/|W_G(K)| \arrow[r] & \mathbb F_p
      \end{tikzcd}
\]
through a ring with trivial $\pi^{\et}_{\leq 1}$.
Since the marks homomorphism is a faithfully flat ring map
\[ A(G)/\cat F\cup K\longrightarrow \prod_{(H)\notin\cat F\cup K}\mathbb Z,\]
the $\mathbb F_p$-point $b'$ lifts to some $\overline{\mathbb F}_p$-valued point of some factor $\mathbb Z$, and since every $\overline{\mathbb F}_p$-valued point of $\mathbb Z$ is $\mathbb F_p$-rational and $\pi^{\et}_{\leq 1}(\mathbb Z)=*$, we conclude.
\item[(2)] For the second proof, we use that the map $b\colon A(G)/\cat F \cup K \to \mathbb Z/|W_G(K)|$ factors through the augmentation map $A(W_G(K))/\{e\} \to \mathbb Z/|W_G(K)|$ so up to renaming we may instead consider the augmentation map $A(G)/[G] \to \mathbb Z/|G|$. Given the profinite fundamental groupoid of the latter recalled above, it suffices to prove the same result for the composite $A(G)/[G]\to \mathbb Z/|G|\to \mathbb Z/p$ for every prime $p$ dividing $|G|$. Finally for this, we let $H$ be a nontrivial $p$-subgroup of $G$. For every finite $H$-set $X$, and in particular for (restrictions of) $G$-sets, we have $|X| = |X^H|$ modulo $p$, and $G^H = \emptyset$, so that the composite $A(G)/[G] \to \mathbb Z/p$ factors through the map $|(-)^H| : A(G)/[G ]\to \mathbb Z$, and $\mathbb Z$ has no étale $\pi_1$, again by Minkowski's theorem, which proves the claim. 

\end{enumerate}
\end{proof}

To guide the reader, we remark that keeping track of exactly which component $b'$ factors through in the above proof will be the most laborious part of the algorithm computing $\pi_{\leq 1}^{\et}(A(G))$. 
\section{The \'{e}tale fundamental groupoid of the Burnside ring}
\label{sec:burnside}

The result of the previous section motivates us to give an explicit description of the \'etale fundamental groupoid of the Burnside ring $A(G)$ for every finite group $G$. While this could have been done sixty years ago, we are unaware of any prior accounts, so we take this opportunity to work this out. To simplify notation, we write 
\[
\mathfrak F(n):=B\widehat{F(n)}
\]
for the groupoid associated to the profinite completion of the free group $F(n)$ on $n\ge 0$ generators. Our first result then states that $\pi^{\et}_{\leq 1}(A(G)/\cat F)$ is the profinite completion of a disjoint union of bouquets of circles.

\begin{Thm}\label{thm:structure}
Let $G$ be a finite group and $\cat F$ a proper family of subgroups of $G$. Then 
\[ \pi^{\et}_{\leq 1}(A(G)/\cat F)\simeq \bigsqcup_{i=1}^{n(\cat F)}\mathfrak F(\alpha(\cat F,i))\]
for uniquely determined integers $n(\cat F)\ge 1$ and $\alpha(\cat F,i)\ge 0$.
\end{Thm}

\begin{Rem}
   The integer $n(\cat F)$ equals the number of connected components of $\Spec(A(G)/\cat F)$.  Indeed, the connected components of the \'etale fundamental groupoid agree with the poset of subobjects of the terminal object of the \'etale topos of $A(G)/\cat F$ by \cite[Proposition 5.45]{Mathew2016}. Since the \'etale topology is subcanonical, this agrees with the subobjects of $A(G)/\cat F$ in the Zariski topos, which in turn agree with the set of connected components of the Zariski spectrum.
\end{Rem}
If the group $G$ is solvable, then we know by a result of Dress \cite{Dress} that $\Spec(A(G))$ is connected. In this case we can give the following closed formula for the number of generators of the \'etale fundamental groupoid. 

\begin{Thm}\label{thm:solvable}
If $G$ is solvable, then $\pi_{\leq 1}^{\et}(A(G))$ is the free profinite group on a number of generators given by
\[ 
\sum\limits_{(K)\subsetneq (G)}\biggl( |\{ p \mid p>0 \text{ prime: }p \, |\,|W_G(K)|\}|-1\biggr),\]
where the sum runs over the conjugacy classes of proper subgroups.
\end{Thm}
\begin{Exa}\label{ex-p-group}
    For every $p$-group $G$, the \'etale fundamental groupoid $\pi_{\leq 1}^{\et}(A(G))$ is trivial. This recovers \cite[Corollary 9.10]{Separablepaper}. Conversely, if $\pi_{\leq 1}^{\et}(A(G))$ is trivial, then $G$ must be solvable but we have not been able to find a group theoretic chracterization for the triviality of $\pi_{\leq 1}^{\et}(A(G))$.\footnote{A computation reveals that of the $292$ isomorphism classes of (solvable) groups of order less than $60$ and not a prime power, 26 have trivial $\pi_{\leq 1}^{\et}(A(G))$.} 
\end{Exa}
\begin{Exa}\label{ex:divisors}
    Let $G=C_N$ be the cyclic group with $N$ elements. Decompose $N$ into its prime factorization 
    \[
    N=\prod_{i=1}^n p_i^{\alpha_i}
    \]
    for $n\ge 0$ and $\alpha_i\ge 1$. Then $\pi_{\leq 1}^{\et}(A(C_N))$ is the free profinite group on
\[ \sum\limits_{\underline\beta}\bigl( |\{ i\, |\, \alpha_i>\beta_i\}|-1\bigr)\]
generators, where the sum extends over all sequences of integers $\underline{\beta}=(\beta_1,\ldots,\beta_n)$ with $0\leq\beta_i\leq\alpha_i$ and $\underline{\beta}\neq\underline{\alpha}$.
\end{Exa}
The following example shows that any finitely generated free profinite group can appear as a Galois groupoid.

\begin{Exa}\label{ex-cyclic-cpq}
    For all $n\ge 0$ and distinct primes $p\neq q$ the \'etale fundamental groupoid of $A(C_{p^{n+1}q})$ is the free profinite group on $n$ generators. This is immediate from \Cref{ex:divisors}. 
\end{Exa}
The rest of this section is devoted to proving the above results. We start by collecting a few preliminaries on profinite groupoids. 
\begin{Rec}
    Let $\mathcal{S}_\pi^\wedge$ denote the $\infty$-category of profinite spaces which is defined as the $\infty$-category of pro-objects in $\pi$-finite spaces, see \cite[Definition E.0.7.11]{SAG}. We identify the $\infty$-category of profinite groupoids $\mathrm{Pro}(\mathsf{Grp}_{\fin})$ with the subcategory of $1$-truncated objects in $\mathcal{S}_\pi^\wedge$. Then there are colimit preserving functors 
    \[
    (-)^\wedge_{\pi} \colon \mathcal{S} \to \mathcal{S}^\wedge_\pi \quad \mathrm{and} \quad \tau_{\leq 1} \colon \mathcal{S}^\wedge_\pi \to \mathrm{Pro}(\mathsf{Grp}_{\fin})
    \]
    respectively given by the profinite completion functor \cite[Proposition E.4.4.1]{SAG} and the $1$-truncation functor \cite[Example E.4.1.3]{SAG}. All pushouts of profinite groupoids we will wish to consider turn out to lie in the essential image of $\tau_{\leq 1}\circ (-)_\pi^\wedge$, and the special values we need to know are 
    \[
    (*)_\pi^\wedge=* \quad \mathrm{and} \quad  \tau_{\leq 1}((BG)_\pi^\wedge)=B\widehat{G},
    \]
    where $G$ is an arbitrary group with pro-finite completion $\widehat{G}$. 
\end{Rec}
\begin{proof}[Proof of \cref{thm:structure}]
    We show by descending induction on $\cat F$ that $\pi^{\et}_{\leq 1}(A(G)/\cat F)$ is the profinite completion of a finite disjoint union of bouquets of circles. The base case of $\cat F=\mathcal P$ being the family of proper subgroups follows from $A(G)/\mathcal P\simeq\mathbb Z$ and $\pi^{\et}_{\leq 1}(\mathbb Z)=*$ by Minkowski's theorem.
For the induction step we consider the following diagram in which the outer square is the pushout provided by \cref{prop:explicit_po} as is the indicated factorization through $\pi_0$:
\begin{equation}\label{megasquare}
    \begin{tikzcd}
       \pi_{\leq 1}^{\et}(\mathbb Z/|W_G(K)|) \arrow[r]\arrow[d] & \pi_0(\pi_{\leq 1}^{\et}(\mathbb Z/|W_G(K)|))\arrow[r]\arrow[d] & \pi_{\leq 1}^{\et}(A(G)/\cat F\cup K)\arrow[d] \\
        *=\pi_{\leq 1}^{\et}(\mathbb Z) \arrow[r] & * \arrow[r] & \pi_{\leq 1}^{\et}(A(G)/\cat F).
    \end{tikzcd}
    \end{equation}
    We observe that 
    \begin{equation}\label{circle}
    \pi_{\leq 1}^{\et}(\mathbb Z/|W_G(K)|)= \bigsqcup_{ 0<p \,| \,|W_G(K)|} \mathfrak F(1)
    \end{equation}
    is the disjoint union of as many copies of the pro-finitely completed circle as there are prime divisors of $|W_G(K)|$ (possibly none!). The claim \eqref{circle} follows from the Chinese remainder theorem, nil-invariance of the \'etale fundamental groupoid and the computation of the absolute Galois group of a finite field. This implies that the left square is a pushout, remembering that we work in $1$-truncated profinite spaces, hence the right hand square is a pushout as well. Since the claimed form of $\pi_{\leq 1}^{\et}(A(G)/\cat F)$ is clearly stable under attaching finitely many $1$-cells, this concludes the proof.
\end{proof}
In order to prove \cref{thm:solvable} it will be useful to set some notation. 
\begin{Def}
    We organize the invariants in \cref{thm:structure} into an unordered list 
    \[
    L(\cat F)=[\alpha(\cat F,1), \ldots, \alpha(\cat F, n(\cat F))]
    \]
    and define  
    \[ 
    \chi(\cat F):=n(\cat F) - \sum_{i=1}^{n(\cat F)}\alpha(\cat F,i).
    \]
\end{Def}

\begin{Exa}\label{ex-chiP}
    For $\cat P$ the family of proper subgroups of $G$, we have already noted that $\pi_{\leq 1}^{\et}(A(G)/\cat P)=\pi_{\leq 1}^{\et}(\Z)=\ast$ from which we deduce that 
    \[
    n(\cat P)=1 \quad \alpha(\cat P, 1)=0 \quad \chi(\cat P)=1 \quad \mathrm{and} \quad L(\cat P)=[0].
    \]
\end{Exa}
\begin{Rem}\label{rem-dress}
    Note that $\chi(\cat F)$ determines $L(\cat F)$ if and only if $n(\cat F)=1$. In general, we do not know a simple characterization for when this happens, but $n(\emptyset)=1$ is equivalent to $\Spec(A(G))$ being connected, i.e. $G$ being solvable.
\end{Rem}
\begin{Not}\label{notation}
    We set 
    \[
        \mathsf P(W_G(K))\coloneqq \pi_0(\pi_{\leq 1}^{\et}(\mathbb Z/|W_G(K)|))=\{p \text{ prime}\mid\, 0 < p \, | \, |W_G(K)|\}\\
    \]
    using \eqref{circle}, and write 
    \[
    f^G_{\cat F, K} \colon \mathsf P(W_G(K)) \to\pi_0(\pi_{\leq 1}^{\et}(A(G)/\cat F\cup K))
    \]
    for the effect on connected components of the map 
    \[ 
    \pi^{\et}_{\leq 1}(\mathbb Z/|W_G(K)|)\longrightarrow\pi_{\leq 1}^{\et}(A(G)/\cat F\cup K)
    \]
from \cref{prop-alg-galois-pullback}.
\end{Not}
\begin{Prop}\label{prop:euler_add}
    Let $\cat F$ be a family of subgroups of $G$ and $K\subseteq G$ a subgroup minimally not in $\cat F$. 
    \begin{enumerate}
    \item Then the list $L(\cat F)$ consists of $\alpha(\cat F \cup K, i)$ for all $ i \not \in \im(f_{\cat F, K}^G)$, and the number 
    \[ 
    |\mathsf P(W_G(K))|-|\im(f_{\cat F, K}^G)|+\sum_{i\in\im(f_{\cat F, K}^G)}\alpha (\cat F\cup K,i).
    \]
\item $\chi(\cat F) = \chi(\cat F\cup K)+1-|\mathsf{P}(W_G(K))|.$
\end{enumerate}
\end{Prop}
\begin{proof}
    Part (b) is a direct computation using (a) which we leave to the reader.
    
    We focus on part (a) which is a fine print of the above proof of \cref{thm:structure}. The right hand pushout square of \eqref{megasquare} using the notation introduced in \cref{notation} can be written as 
    \begin{equation}\label{square2}
    \begin{tikzcd}
          \mathsf P(W_G(K)) \arrow[r, "f_{\cat F,K}^G"] \arrow[d] & \pi_{\leq 1}^{\et}(A(G)/\cat F\cup K)\arrow[d]\\
          \ast \arrow[r] & \pi_{\leq 1}^{\et}(A(G)/\cat F).
    \end{tikzcd}
    \end{equation}
    
    We can factors the top horizontal map of the previous square as the composite 
    \[
    \mathsf P(W_G(K)) \xrightarrowdbl{f_{\cat{F}, K}^G}  \im(f_{\cat F, K}^G) \hookrightarrow \pi_0(\pi_{\leq 1}^{\et}(A(G)/\cat F\cup K)) \xhookrightarrow{\iota}\pi_{\leq 1}^{\et}(A(G)/\cat F\cup K)
    \]
    where the map $\iota$ is the inclusion of some chosen based points. For notational convenience we set 
    \[
    A= \mathsf P(W_G(K))\quad X=\im(f_{\mathcal{F}, K}^G) \quad B= \pi_0(\pi_{\leq 1}^{\et}(A(G)/\cat F\cup K))
    \]
    so that the above composite reads
    \[
    A \xrightarrow{f_{\cat F, K}^G} X \hookrightarrow B \xhookrightarrow{\iota} \pi_{\leq 1}^{\et}(A(G)/\cat F\cup K).
    \]
    The pushout square \eqref{square2} can be obtain by pasting the following diagrams
    \[
    \begin{tikzcd}
       \bigsqcup_{a\in A}\ast  \arrow[r,two heads, "f_{\mathcal{F}, K}^G"]\arrow[d] & X\arrow[d] \arrow[r, hook]&   \bigsqcup_{b \in B} \ast \arrow[d]\\
        *\arrow[r]& \mathfrak F(|A|-|X|) \arrow[r] & \mathfrak F(|A|-|X|)  \bigsqcup ( \bigsqcup_{b \in B\backslash X} \ast )
    \end{tikzcd}
    \]
    and 
    \[
    \begin{tikzcd}
        \bigsqcup_{b \in B}\ast \arrow[r, hook, "\iota"]\arrow[d] & \pi_{\leq 1}^{\et}(A(G)/\cat F\cup K)\arrow[d] \\
        \mathfrak F(|A|-|X|) \sqcup  ( \bigsqcup_{b \in B \backslash X} \ast)   \arrow[r] & \pi_{\leq 1}^{\et}(A(G)/\cat F).
    \end{tikzcd}
    \]
 We now observe that the first two squares are pushouts and conclude by the pasting lemma that the third is too. Taking into account the induction hypothesis we can rewrite this third square as
\[
    \begin{tikzcd}
         \bigsqcup_{b \in B}\ast \arrow[r, hook, "\iota"]\arrow[d] &   \bigsqcup_{b\in  B} \mathfrak F(\alpha(\cat F\cup K, b))\arrow[d]\\
           \mathfrak F(|A|-|X|) \sqcup  ( \bigsqcup_{b \in B \backslash X} \ast)    \arrow[r] & \pi_{\leq 1}^{\et}(A(G)/\cat F).
    \end{tikzcd}
    \]
The cospan of which decomposes into a disjoint union indexed by $X$ and $B\setminus X$, and the entire pushout decomposes accordingly as 
    \[
    \begin{tikzcd}
       ( \bigsqcup_{b\in X}*)  \sqcup ( \bigsqcup_{b\in B\setminus X}*)\arrow[r, hook]\arrow[d] &  ( \bigsqcup_{b\in X} \mathfrak F(\alpha(\cat F\cup K, b))) \sqcup ( \bigsqcup_{b\in B\setminus X} \mathfrak F(\alpha(\cat F\cup K, b)))\arrow[d]\\
        \mathfrak F (|A|-|X|) \sqcup ( \bigsqcup_{b\in B\setminus X} *)\arrow[r] & \mathfrak F(|A|-|X|+\sum_{b\in X}\alpha(\cat F\cup K, b)) \sqcup ( \bigsqcup_{b\in B\setminus X} \mathfrak F(\alpha(\cat F \cup K, b))).
    \end{tikzcd}
    \]
    The numbers appearing in $L(\cat F)$ are then given by the $\alpha(\cat F\cup K, b)$ for $b \in B \backslash X$, and
    \[
    |A| -|X| +\sum_{b \in X} \alpha(\cat F \cup K, b)
    \]
    from which the claim follows.
\end{proof}
\begin{Rem}
    If one is only interested in point (b) of the previous Proposition (e.g. in the case of solvable groups where $\chi(\cat F)$ determines the whole Galois groupoid), there is another slightly more conceptual way to obtain it: $\chi(\cat F)$ can be viewed as profinite Euler characteristic of $\pi_{\leq 1}^{\et}(A(G)/\cat F)$, defined e.g. by taking continuous cohomology with coefficients in any finite ring. This profinite Euler characteristic is then additive along pushouts of profinite spaces. The key point is then that while the square \eqref{megasquare} is only a pushout square of profinite $1$-groupoids, square \eqref{square2} is automatically a pushout square of profinite groupoids, with no $1$-truncation needed.\\ 
Since it would take more time to provide details on the relevant profinite Euler characteristic, and (b) follows from (a) in which we are also interested, we have opted to stick to the current proof. 
\end{Rem}

\begin{proof}[Proof of \cref{thm:solvable}]
Since $G$ is solvable we know that $n(\emptyset)=1$, and so by \cref{thm:structure} the \'etale fundamental groupoid is given by $\mathfrak F(\alpha(\emptyset, 1))$. Therefore we need to verify that the integer $\alpha(\emptyset,1)$ is given as in the theorem. On one side we know that 
\begin{equation}\label{formula1}
\chi(\emptyset)\coloneqq n(\emptyset)-\alpha(\emptyset,1)=1 - \alpha(\emptyset,1).
\end{equation}
On the other hand, we can choose a filtration of families 
\[
\emptyset \subseteq \cat F_1 \subseteq \cat F_2 \subseteq \ldots \subseteq \cat F_s= \cat P
\]
where each $\cat F_{i+1} \backslash \cat F_i$ consists only of a single $G$-conjugacy class $(K_{i+1})$ and $\cat P$ is the family of proper subgroups. Using \cref{prop:euler_add}(b) repeatedly, we find:
\begin{align*}
    \chi(\emptyset)& = \chi(\cat F_1) +1 - |\mathsf P(W_G(K_1))| \\
                   & =(\chi(\cat F_2) +1 - |\mathsf P(W_G(K_2))|) +1 -  |\mathsf P(W_G(K_1))|\\
                   & = \ldots \\
                   & =\chi(\cat F_s) +1 - |\mathsf P(W_G (K_s))| + \ldots + 1 - |\mathsf P(W_G(K_2))| +1 -  |\mathsf P(W_G(K_1))|\\
                   & =1 +1 - |\mathsf P(W_G (K_s))| + \ldots + 1 - |\mathsf P(W_G(K_2))| +1 -  |\mathsf P(W_G(K_1))|\\
\end{align*}
using that $\chi(\cat F_s)=\chi(\cat P)=1$ by \cref{ex-chiP}. Comparing this final equation with \eqref{formula1} we obtain the require formula for $\alpha(\emptyset,1)$.

\end{proof}
\section{The \'etale fundamental groupoid of the Burnside ring using GAP}\label{sec-diagram}
In this section, we discuss an algorithm for calculating the \'etale fundamental groupoid of the Burnside ring from the table of marks. We explain this algorithm with a few small examples and then implement it in GAP. For a sample of our calculations, we refer the reader to \cref{tab:group_L}.

\subsection{The algorithm}
The following documented meta code computes the \'etale fundamental groupoid of $A(G)$ from the table of marks of $G$ using GAP.
\begin{itemize}
\item[(0)] import a table of marks tom of a group $G$, extract its size $s$ (the number of subgroups of $G$ up to conjugacy) and the group order $g$ and for every prime $p>0$ dividing $g$, create $C(p):=\mathrm{CyclicExtensionsTom}(\mathrm{tom},p).$
\end{itemize}
The Zariski spectrum of $A(G)$ consists of $s$ many copies of $\Spec(\mathbb Z)$ glued along closed points. For every prime $p$, the equivalence relation among the components of being glued along a point of residue characteristic $p$ is recorded in $C(p)$.
\begin{Exa}
    Let $C_6$ be the cyclic group of order $6$. Using GAP we can extract its table of marks $\mathrm{tom}$:
\[
\begin{tabular}{|l|c|c|c|c|}
 \hline
  1   & 6 & 0 & 0 & 0  \\
   \hline
  $C_2$   & 3 & 3 & 0 & 0 \\
   \hline
 $C_3$   & 2 & 0 & 2 & 0 \\
   \hline
  $C_6$   & 1 & 1 & 1 & 1 \\
  \hline
\end{tabular}
\]
Notice that the diagonal entries of the table of marks are precisely the orders of the Weyl groups. Since the only prime numbers diving these diagonal entries are $2$ and $3$, we just need to calculate the gluing with respect to these primes using the command \text{CyclicExtensionsTom(tom, p)} in GAP:
\begin{align*}
    \text{CyclicExtensionsTom(tom, 2)} & = [[1,2], [3,4]]\\
    \text{CyclicExtensionsTom(tom, 3)} & = [[1,3], [2,4]]\\
\end{align*}
We can draw the above gluing in the following way:
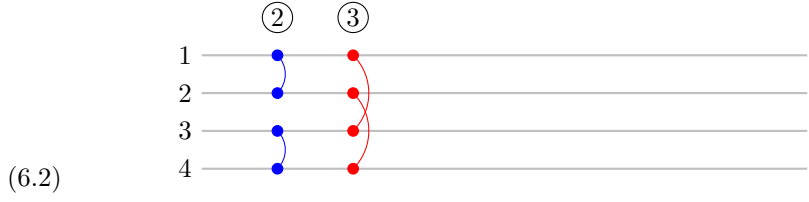
\begin{equation}\label{gluingC6}
\begin{tikzpicture}
    \foreach \x in {1,2,...,4}{
        \draw[gray!50, thick] plot [smooth] coordinates {(8,-\x/2) (0,-\x/2)} node [black, left] {\x};
        }
    \filldraw[blue] (1,-1/2) circle (2pt);
    \filldraw[blue] (1,-2/2) circle (2pt);
    \filldraw[blue] (1,-3/2) circle (2pt);
    \filldraw[blue] (1,-4/2) circle (2pt);
    \draw [blue] (1,-1/2) to [bend left=45] (1,-2/2);
    \draw [blue] (1,-3/2) to [bend left=45] (1,-4/2);
    \draw (1,0) circle [radius=0.2] node {$2$};
    \filldraw[red] (2,-1/2) circle (2pt);
    \filldraw[red] (2,-2/2) circle (2pt);
    \filldraw[red] (2,-3/2) circle (2pt);
    \filldraw[red] (2,-4/2) circle (2pt);
    \draw [red] (2,-1/2) to [bend left=45] (2,-3/2);
     \draw [red] (2,-2/2) to [bend left=45] (2,-4/2);
    \draw (2,0) circle [radius=0.2] node {$3$};
\end{tikzpicture}   
\end{equation}
We then consider a filtration of families of $C_6$:
\[
\emptyset \subseteq \cat F_1 \subseteq \cat F_2 \subseteq \cat F_3 \subseteq \cat F_4 \subseteq \mathrm{All}
\]
where $\cat F_4$ is the family of proper subgroups, and we move up by one family by adding a single group at the time. We can describe the \'etale fundamental groupoid of $A(C_6)$ iteratively using the pushout diagram \eqref{square2}. At every step we are gluing in a copy of $\pi_{\leq 1}^{\et}(\Z)=\ast$ and the gluing witness by the map $f_{\cat F,K}^G$ in \eqref{square2} is precisely recorded in \eqref{gluingC6}.
\end{Exa}
We carry on with our algorithm:
\begin{itemize}
\item[(1)] create an integer $c$, an ordered list of integers $L$ and an ordered list of ordered lists of integers $C$, and set $c:=s$, $L:=[0]$, $C:= [ [c] ]$.
\end{itemize}
The number $c$ indicates the number of the last copy of $\Spec(\mathbb Z)$ that has been glued in to build $\Spec(A(G))$. The list $L$ is the list from last section, which we want to compute iteratively using the formula in \cref{prop:euler_add}, (a). The list $C$ is the equivalence relation among the copies of $\Spec(\Z)$ glued in of belonging to the same component of the current $\Spec(A(G)/\cat F_c)$.
\begin{itemize}
\item[(2)] if $c=1$ then STOP ($L$ is the result) else
\item[(3)] create the set $P$ of positive prime divisors of the $c-1$ diagonal entry of tom.
\end{itemize}
These are the prime divisors of the Weyl group $W_G(K)$.
\begin{itemize}
\item[(4)] if $|P|=0$ then make 
\[
L := \text{Concatenation}( [ 0 ], L ) \quad \mathrm{and}\quad C = \text{Concatenation}([c-1], C)
\]
and go to (13), else
\item[(5)] For $p \in P$, let $Ep$ be the entry of $C(p)$ containing $c-1$.
\item[(6)]Create the subset
\[ 
I:=\bigcup_{p\in P} \left\{ x\in C\mid x\cap Ep\neq\emptyset \right\} 
\]
\end{itemize}
This $I$ counts the gluing occurring at this step, i.e., the components of $A(G)/\cat F_c$ in the image of $f^K_{\cat F_c}$.
\begin{itemize}
\item[(7)] Let $A$ be the list of indexes $j$ such that the $j$-entry of $C$ is in $I$. Let $B$ be the list of indexes $j$ such that the $j$-entry of $C$ is not in $I$.
\item[(8)] Let $L^A=\mathrm{Sublist}( L, A )$ be the sublist of $L$ described by $A$. Similarly,
let $L^B=\mathrm{Sublist}( L, B )$ be the sublist of $L$ described by $B$.
\item[(9)] Let $\mathrm{sum}(L^A)$ be the sum of all entries in $L^A$. Set $N=|P|-|I|+\mathrm{sum}(L^A)$.
\item[(10)] Redefine 
\[
L\coloneqq\mathrm{Concatenation}([N], L^B).
\]
\end{itemize}
This is the formula for $L$ given in \cref{prop:euler_add}(a).
\begin{itemize}
\item[(11)] Create a list $C(I)$ containing $c-1$ and all the elements of $I$. Create a sublist $NC(I)$ of $C$ with all the elements of $C$ not in $I$.
\item[(12)] Redefine 
\[
C:=\mathrm{Concatenation}(C(I), NC(I)).
\]
\item[(13)] decrease $c$ by 1 and go to (2).
\end{itemize}

For concreteness, we run through this algorithm by hand in two examples. 

\subsection{Cyclic group of order $6$}
We record the table of marks
\[
\begin{tabular}{|l|c|c|c|c|}
 \hline
  1   & 6 & 0 & 0 & 0  \\
   \hline
  $C_2$   & 3 & 3 & 0 & 0 \\
   \hline
 $C_3$   & 2 & 0 & 2 & 0 \\
   \hline
  $C_6$   & 1 & 1 & 1 & 1 \\
  \hline
\end{tabular}
\]
and the gluing diagram 
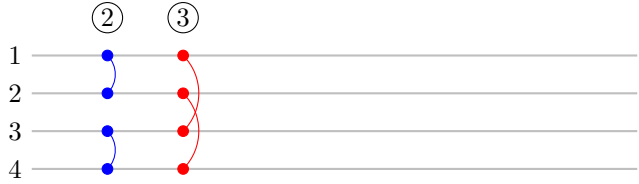
\[
\begin{tikzpicture}
    \foreach \x in {1,2,...,4}{
        \draw[gray!50, thick] plot [smooth] coordinates {(8,-\x/2) (0,-\x/2)} node [black, left] {\x};
        }
    \filldraw[blue] (1,-1/2) circle (2pt);
    \filldraw[blue] (1,-2/2) circle (2pt);
    \filldraw[blue] (1,-3/2) circle (2pt);
    \filldraw[blue] (1,-4/2) circle (2pt);
    \draw [blue] (1,-1/2) to [bend left=45] (1,-2/2);
    \draw [blue] (1,-3/2) to [bend left=45] (1,-4/2);
    \draw (1,0) circle [radius=0.2] node {$2$};
    \filldraw[red] (2,-1/2) circle (2pt);
    \filldraw[red] (2,-2/2) circle (2pt);
    \filldraw[red] (2,-3/2) circle (2pt);
    \filldraw[red] (2,-4/2) circle (2pt);
    \draw [red] (2,-1/2) to [bend left=45] (2,-3/2);
     \draw [red] (2,-2/2) to [bend left=45] (2,-4/2);
    \draw (2,0) circle [radius=0.2] node {$3$};
\end{tikzpicture}
\]
for the cyclic group of order $6$.
Since the above diagram consists of four rows, the algorithm will consist of four iterations:
\begin{itemize}
    \item[(1)] We set $L=[0]$ and $C:=[[4]]$. This $L$ encodes that $\pi_{\leq 1}^{\et}(A(C_6)/\cat P)=
    \ast$ and $C$ is the list of connected components of $\pi_{\leq 1}^{\et}(A(C_6)/\cat P)$: one component corresponding to the fourth horizontal line in the above schematic. 
    \item[(2)] Look at the third row of the table of marks. The diagonal term in this row is $2$, so the set $P$ is prime divisors of this entry is a singleton. The gluing diagram shows that the third line is glued along $2$ with the fourth line. We record this as follows:
    \[
    I=\{a \in C \mid a \cap [3,4] \not= \emptyset \}=\{[4]\}.
    \]
    We note that $C -I $ is empty. It follows that $L^B$ is empty and $L^A=[0]$. We then update the $C$ to $C=[[3,4]]$. And set $L=[|P|-|I|+ 0]=[0]$.
    \item[(3)] Look at the second row of the table of marks. The diagonal term in this row is $3$, so the set $P$ is prime divisors of this entry is a singleton. The gluing diagram shows that the second line is glued along $3$ with the fourth line. Again we record 
     \[
    I=\{a \in C \mid a \cap [2,4]\not =\emptyset \}=\{[3,4]\}.
    \]
    We calculate $I - C$. This is empty so we update the $C$ to $C=[[2,3,4]]$. And set $L=[|P|-|I|+0]=[0]$. 
    \item[(4)] Look at the first row of the table of marks. The diagonal term in this row is $6$, and so now the set $P$ of prime divisors has two elements. There is some glueing at the prime $2$ and $3$, so we record it 
    \[
    I = \{\{a \in C \mid a \cap [1,2]\not =\emptyset \}\} \cup \{\{a \in C \mid a \cap [1,3]\not =\emptyset \}\} =\{[2,3,4]\}
    \]
    We calculate $I -C$ which is empty, so we update $C=[[1,2,3,4]]$ and $L=[|P|-|I|+0]=[1]$. 
\end{itemize}
Thus the \'etale fundamental groupoid is the profinite completion of a circle. 
\subsection{Symmetric group $S_3$}
Using GAP we can extract its table of marks $\mathrm{tom}$:
\[
\begin{tabular}{|l|c|c|c|c|c|c|c|c|c|}
 \hline
  $1$   & 6 & 0 & 0 & 0   \\
   \hline
  $C_2$  & 3 & 1 & 0 & 0  \\
   \hline
 $C_3$  & 2 & 0 & 2 & 0  \\
   \hline
  $S_3$:  & 1 & 1 & 1 & 1  \\
  \hline
\end{tabular}
\]
Since the only prime numbers diving these diagonal entries are $2$ and $3$, we just need to calculate the gluing with respect to these primes using GAP:
\begin{align*}
    \text{CyclicExtensionsTom(tom, 2)} & = [[1,2], [3,4]]\\
    \text{CyclicExtensionsTom(tom, 3)} & = [[1,3], [2], [4]]\\
\end{align*}
We can draw the above gluing in the following way:
\[
\begin{tikzpicture}
    \foreach \x in {1,2,...,4}{
        \draw[gray!50, thick] plot [smooth] coordinates {(8,-\x/2) (0,-\x/2)} node [black, left] {\x};
        }
    \filldraw[blue] (1,-1/2) circle (2pt);
    \filldraw[blue] (1,-2/2) circle (2pt);
    \filldraw[blue] (1,-3/2) circle (2pt);
    \filldraw[blue] (1,-4/2) circle (2pt);
    \draw [blue] (1,-1/2) to [bend left=45] (1,-2/2);
    \draw [blue] (1,-3/2) to [bend left=45] (1,-4/2);
    \draw (1,0) circle [radius=0.2] node {$2$};
    \filldraw[red] (2,-1/2) circle (2pt);
   
    \filldraw[red] (2,-3/2) circle (2pt);
    
    \draw [red] (2,-1/2) to [bend left=45] (2,-3/2);

    \draw (2,0) circle [radius=0.2] node {$3$};
\end{tikzpicture}
\]
The algorithm will consist of four iterations:
\begin{itemize}
   \item[(1)] We set $L=[0]$ and $C:=[[4]]$. 
    \item[(2)] The diagonal term in the third row is $2$, so the set $P$ of prime divisors is a singleton. Now we have one glueing between the row $3$ and $4$ so we get $C=[[3,4]]$ and $L=[0]$.
    \item[(3)] The diagonal term in second row is $1$, so the set $P$ is prime divisors of this entry is empty. Therefore we set $L=[0, 0]$ and $C=[[2], [3,4]]$.
    \item[(4)] The diagonal term in the first row is $6$, and so now the set $P$ of prime divisors has two elements. There are exactly two gluing so we update $L=[0]$ and $C=[[1,2,3,4]]$. 
\end{itemize}
The \'etale fundamental groupoid is therefore trivial. 
\subsection{Alternating group $A_5$}
Using GAP we can extract its table of marks $\mathrm{tom}$:
\[
\begin{tabular}{|l|c|c|c|c|c|c|c|c|c|}
 \hline
  1:   & 60 & 0 & 0 & 0 & 0 & 0 & 0 & 0 & 0  \\
   \hline
  2:   &30 & 2 & 0 & 0 & 0 & 0 & 0 & 0 & 0   \\
   \hline
 3:   & 20 & 0 & 2 & 0 & 0 & 0 & 0 & 0 & 0  \\
   \hline
  4:   & 15 & 3 & 0 & 3 & 0 & 0 & 0 & 0 & 0   \\
  \hline
  5: & 12 & 0 & 0 & 0 & 2 & 0 & 0 & 0 & 0  \\
  \hline
  6: & 10 & 2 & 1 & 0 & 0 & 1 & 0 & 0 & 0  \\
  \hline
  7: & 6 & 2 & 0 & 0 & 1 & 0 & 1 & 0 & 0  \\
  \hline
  8: & 5 & 1 & 2 & 1 & 0 & 0 & 0 & 1 & 0  \\
  \hline
  9: & 1 & 1 & 1 & 1 & 1 & 1 & 1 & 1 & 1  \\
  \hline
\end{tabular}
\]
where the first column denotes the conjugacy classes of subgroups, with the first row being the trivial group, and the last row being $A_5$. The diagonal entries of the table of marks are precisely the orders of the Weyl groups, which agrees with the prime divisors of $|A_5|$. We calculate the gluing for the primes numbers diving the diagonal entries of the table of marks:
\begin{align*}
    \text{CyclicExtensionsTom}(\mathrm{tom},2)& = [ [ 1, 2, 4 ], [ 3, 6 ], [ 5, 7 ], [ 8 ], [ 9 ] ]\\
    \text{CyclicExtensionsTom}(\mathrm{tom},3)& = [ [ 1, 3 ], [ 2 ], [ 4, 8 ], [ 5 ], [ 6 ], [ 7 ], [ 9 ] ]\\
    \text{CyclicExtensionsTom}(\mathrm{tom},5)& = [ [ 1, 5 ], [ 2 ], [ 3 ], [ 4 ], [ 6 ], [ 7 ], [ 8 ], [ 9 ] ]\\
\end{align*}
and we depicted the gluing as before via the diagram:
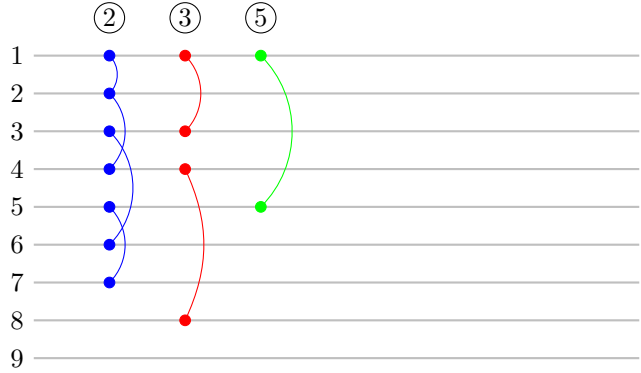
\[
\begin{tikzpicture}
    \foreach \x in {1,2,...,9}{
        \draw[gray!50, thick] plot [smooth] coordinates {(8,-\x/2) (0,-\x/2)} node [black, left] {\x};
        }
    \filldraw[blue] (1,-1/2) circle (2pt);
    \filldraw[blue] (1,-2/2) circle (2pt);
    \filldraw[blue] (1,-3/2) circle (2pt);
    \filldraw[blue] (1,-4/2) circle (2pt);
    \filldraw[blue] (1,-5/2) circle (2pt);
    \filldraw[blue] (1,-6/2) circle (2pt);
    \filldraw[blue] (1,-7/2) circle (2pt);
    \draw [blue] (1,-1/2) to [bend left=45] (1,-2/2);
    \draw [blue] (1,-2/2) to [bend left=45] (1,-4/2);
    \draw [blue] (1,-3/2) to [bend left=45] (1,-6/2);
    \draw [blue] (1,-5/2) to [bend left=45] (1,-7/2);
    \draw (1,0) circle [radius=0.2] node {$2$};
    \filldraw[red] (2,-1/2) circle (2pt);
    \filldraw[red] (2,-3/2) circle (2pt);
    \filldraw[red] (2,-4/2) circle (2pt);
    \filldraw[red] (2,-8/2) circle (2pt);
    \draw [red] (2,-1/2) to [bend left=45] (2,-3/2);
    \draw [red] (2,-4/2) to [bend left=25] (2,-8/2);
    \draw (2,0) circle [radius=0.2] node {$3$};
    \filldraw[green] (3,-1/2) circle (2pt);
    \filldraw[green] (3,-5/2) circle (2pt); 
    \draw [green] (3,-1/2) to [bend left=45] (3,-5/2);
   
    \draw (3,0) circle [radius=0.2] node {$5$};
\end{tikzpicture}
\]
The rank of the table of marks is nine, so the algorith will have nine steps:
\begin{enumerate}
    \item[(1)] As usual we start with $L=[0]$ and $C=[[9]]$. 
    \item[(2)] Look at the eighth row which has diagonal term given by $1$. In this case the set $P$ of prime divisor of this entry is empty, so according to the algorithm we set $L=[0,0]$ and $C=[[8], [9]]$.
    \item[(3)] Look at the seventh row and here again we have a $1$ in the diagonal. Thus following the above we get $L=[0, 0, 0]$ and $C=[[7], [8], [9]]$.
    \item[(4)] Look at the sixth row where we have a $1$ in the diagonal. Thus we redefine $L=[0, 0, 0, 0]$ and $C=[[6], [7], [8], [9]]$.
    \item[(5)] Look at the fifth row, now we have a $2$ in the diagonal, so the set $P$ of prime divisors is a singleton. Now we have one glueing between the row $5$ and $7$ and in this case we get $C=[[5,7],[6],[8], [9]]$ and $L=[0,0,0,0]$.
    \item[(6)] Look at the fourth row, now we have a $3$ in the diagonal. We have only one gluing between the fourth and  eight row, so we get $C=[[4,8], [5,7], [6], [9]]$ and $L=[0,0,0,0]$.
    \item[(7)] Look at the third row, now we have a $2$ in the diagonal. We have only one glueing at the prime $2$ between the third and sixth row, so $C=[[3,6], [4,8], [5,7], [9]]$ and $L=[0,0,0,0]$.
    \item[(8)] Look at the second row, now we have a $2$. There is one gluing between the second and fourth row, so $C=[[2,4,8], [3,6], [5,7], [9]]$ and $L=[0,0,0,0]$.
    \item[(9)] Look at the first row, now we have $60$. Now $P$ has three elements but there are also three gluing involving the first row. In this case we get $C=([2,4,8,3,6,5,7],[9])$ and $L=[0,0]$.
\end{enumerate}
Thus the \'etale fundamental groupoid is just two points.

\subsection{GAP code}
We have implemented the above algorithm in GAP and the code is available at \url{https://sites.google.com/view/lucapol/}. We ran the code for several finite groups and the results are displayed in \cref{tab:group_L}.

\begin{landscape}
\begin{table}[htbp]
\centering
\caption{Values of $L$ for various finite groups}
\label{tab:group_L}

\renewcommand{\arraystretch}{1.2}

\begin{tabular}{|c|c|}
    \text{Group} & \text{L} \\
    \hline
     $A_5$ & [0,0] \\
     \hline
     $A_6$ & [0,0,0,0] \\
     \hline
     $A_7$ & [2,0,0,0,0,0,0]\\
     \hline
     $A_8$ & [16,0,0,0,0,0,0,0,0,0,0] \\
     \hline
     $A_9$ & [37,0,0,0,0,0,0,0,0,0,0,0,0] \\
     \hline 
     $A_{10}$ & [72,1,0,0,0,0,0,0,0,0,0,0,0,0,0,0,0,0]\\
     \hline 
     $A_{11}$ & [182,6,1,0,0,0,0,0,0,1,0,0,0,0,0,0,0,0,0,0,0,0,0,0,0,0]\\
     \hline 
     $A_{12}$ & [489,15,6,0,0,0,0,0,0,0,0,0,6,0,0,0,0,0,0,0,0,0,0,0,0,1,0,0,0,0,0,0,0,0,0,0,0,0,0,0,0,0,0,0,0,0]\\
     \hline
     $A_{13}$ & [1012,40,15,1,0,0,0,0,0,0,1,15,1,0,0,0,0,2,0,0,1,0,0,6,2,0,0,0,0,0,0,0,0,0,0,0,1,0,0,0,0,0,0,0,0,0,0,0,0,0,0,0,0]\\
     \hline 
     \hline
    $S_5$ & [1,0] \\
     \hline
     $S_6$ & [6,0,0,0]\\
     \hline 
     $S_7$ & [15,0,0,0,0,0]\\
     \hline 
     $S_8$ & [40,1,0,0,0,0,0,0,0]\\
     \hline
     $S_9$ & [105,2,1,0,0,1,0,0,0,0,0]\\
     \hline
     $S_{10}$ & [255,6,2,0,0,0,0,2,0,0,0,0,1,0,0,0,0]\\
     \hline
     $S_{11}$ & [617,24,6,0,0,0,0,1,6,0,1,0,0,0,2,0,0,0,1,0,0,0,0]\\
     \hline
     $S_{12}$ & [1633,58,24,0,0,0,0,0,1,2,24,0,0,0,2,0,0,0,0,0,6,0,0,0,0,0,0,0,2,0,0,0,0,0,1,0,0,0]\\
     \hline
     $S_{13}$ & [3748,151,58,4,1,0,0,0,6,6,58,4,0,0,8,0,0,4,1,0,24,7,0,0,0,0,0,0,0,0,0,6,0,0,0,0,0,0,0,0,2,0,0,1,0,0,0]\\
     \hline
     \hline
        
     SL$_2(\mathbb F_{3})$ & [1]\\
      \hline
     SL$_2(\mathbb F_{2^2})$& [0,0]\\
      \hline    
     SL$_2(\mathbb F_{5})$ & [2,0]\\
      \hline  
     SL$_2(\mathbb F_{7})$ & [3,0]\\
      \hline  
      \hline
       PSL$_2(\mathbb F_{2^4})$& [2,0,0]\\
       
      \hline  
      PSL$_2(\mathbb F_{5^3})$ & [4,0,0]\\
      \hline
      PSL$_3(\mathbb F_{11})$  & [198,0,0,0,0,1,0,0,1,1,0]\\
      \hline
      PSL$_4(\mathbb F_3)$  &[186,0,0,0,0,0,0,0,0,0,0,0,0,0]\\
      \hline\hline
 $M_{12}$ & [14,0,0,0,0,0,0,0,0,0,0,0] \\
      \hline
      $M_{23}$ & [25,0,0,0,0,0,0,0,0,0,0,0,0,0,0,0,0,0,0,0,0,0,0,0,0,0,0,0]\\
      \hline
      $M_{24}$ & [181,0,0,0,0,0,0,0,0,0,0,0,0,0,0,0,0,0,0,0,0,0,0,0,0,0,0,0,0,0,0,0,0,0,0,0,0,0,0,0 ]\\
      \hline

\end{tabular}
\end{table}
\end{landscape}

\bibliographystyle{alpha}
\bibliography{reference}

\end{document}